\documentclass{article}
\usepackage[affil-it]{authblk}

\usepackage[utf8]{inputenc}
\usepackage[english]{babel}

\usepackage{graphicx}
\usepackage{amssymb}
\usepackage{amsmath}
\usepackage{amsthm}
\usepackage{tikz-cd}
\usepackage[all]{xy}
\usepackage{enumerate}
\usepackage{hyperref}
\usepackage{ytableau}
\ytableausetup{boxsize=1.5em}

\usepackage{csquotes}
\usepackage[backend=bibtex,style=alphabetic]{biblatex}
\bibliography{main}

\numberwithin{equation}{section}

\allowdisplaybreaks

\usepackage{enumitem}
\setlist[itemize]{leftmargin=1.6em}

\newtheorem{theorem}{Theorem}[section]
\newtheorem{corollary}[theorem]{Corollary}
\newtheorem{lemma}[theorem]{Lemma}
\newtheorem{proposition}[theorem]{Proposition}

\DeclareRobustCommand{\qbinom}{\genfrac{\lbrack}{\rbrack}{0pt}{}}
\def\multiset#1#2{\ensuremath{\left(\kern-.3em\left(\genfrac{}{}{0pt}{}{#1}{#2}\right)\kern-.3em\right)}}

\begin{document}

\title{Quantum integer-valued polynomials}
\author{Nate Harman}
\author{Sam Hopkins}
\affil{Department of Mathematics, MIT}
\maketitle
\begin{abstract}
We define a $q$-deformation of the classical ring of integer-valued polynomials which we call the ring of quantum integer-valued polynomials. We show that this ring has a remarkable combinatorial structure and enjoys many positivity properties: for instance, the structure constants for this ring with respect to its basis of $q$-binomial coefficient polynomials belong to~$\mathbb{N}[q]$. We then classify all maps from this ring into a field,  extending a known classification in the classical case where $q=1$.

\end{abstract}

\begin{section}*{Introduction}

Throughout this paper $x$ and $q$ are commuting indeterminates (which we will sometimes specialize) whereas other lowercase letters like $n$, $m$, $i$, $j$, $k$, and~$d$ are numbers and $p$ is always a prime number. We use $\mathbb{N} := \{0,1,2,\ldots\}$ for the set of natural numbers, $\mathbb{Z}$ for the ring of integers, $\mathbb{Q}$ for the field of rational numbers, $\mathbb{F}_{p^m}$ for the finite field of cardinality~$p^m$, and $\mathbb{Z}_p$ the ring of $p$-adic integers. If $R$ is a ring, then~$R[x]$ denotes the univariate polynomial ring in~$x$ over~$R$. Polynomials will be denoted by uppercase letters like~$P(x)$. If $\mathbf{k}$ is a field, then~$\mathbf{k}(q)$ denotes the field of rational expressions in~$q$ over $\mathbf{k}$, the fraction field of~$\mathbf{k}[q]$. For $a,b \in \mathbb{Z}$ set $[a,b] := \{a,a+1,a+2,\ldots,b\}$, which is~$\varnothing$ if~$a > b$.

A polynomial is integer-valued if it takes integer values at all integers. In the context of interpolation theory, integer-valued polynomials have been studied at least since the work of Isaac Newton in the 17th century. But the theory of integer-valued polynomials was first systematically developed in two 1919 papers of P\'{o}lya~\cite{polya1915ganzwertige} and Ostrowski~\cite{ostrowski1919ganzwertige}. Their focus was on finding so-called ``regular bases'' (i.e., bases consisting of one polynomial of each degree) for $\mathbb{Z}$-algebras of integer-valued polynomials with coefficients in various number fields~$K$. When~$K=\mathbb{Q}$, the classical case, a regular basis for the ring of integer-valued polynomials is given by the binomial coefficient polynomials. For more background on integer-valued polynomials, consult the book of Cahen and Chabert~\cite{cahen1997integer}.

In this paper we investigate a $q$-deformation of the classical ring of integer-valued polynomials. Let us briefly summarize our results here. Recall the $q$-numbers $[n]_q := (q^n-1)/(q-1)$, the $q$-factorials $[n]_q! := [n]_q \cdot [n-1]_q \cdots [1]_q$, and the $q$-binomial coefficients~$\qbinom{n}{k}_q := \frac{[n]_q!}{[n-k]_q![k]_q!}$. In this paper we study the ring~$\mathcal{R}_q$ of all polynomials $P(x) \in \mathbb{Q}(q)[x]$ with $P([n]_q) \in \mathbb{Z}[q,q^{-1}]$ for all~$n \in \mathbb{Z}$. We show that this ring has a basis as a~$\mathbb{Z}[q,q^{-1}]$-module consisting of the $q$-binomial coefficient polynomials~$\qbinom{x}{k}$ for $k \in \mathbb{N}$ (Propositions~\ref{prop:qbinom} and~\ref{prop:localization}). These~$\qbinom{x}{k}$ are the unique polynomials in $\mathbb{Q}(q)[x]$ with~\mbox{$\qbinom{[n]_q}{k} = \qbinom{n}{k}_q$} for all~$n \in \mathbb{N}$. From~$\mathcal{R}_q$ we recover the classical ring $\mathcal{R}$ of integer-valued polynomials with rational coefficients and its basis of binomial coefficient polynomials~$\binom{x}{k}$ by specializing~$q:=1$. It is well-known that $\mathcal{R}$ is not Noetherian (see~\cite[pg.~xvii]{cahen1997integer}). Thus~$\mathcal{R}_q$ is also non-Noetherian. Nevertheless the rings~$\mathcal{R}_q$ and~$\mathcal{R}$ have remarkable combinatorial structures and positivity properties. For starters, these rings come with the following maps:
\begin{itemize}
\item a shift operator $S\colon \mathcal{R}_q \to \mathcal{R}_q$ defined by $S(x):=qx+1$ with inverse $S^{-1}\colon \mathcal{R}_q \to \mathcal{R}_q$ defined by $S^{-1}(x):=q^{-1}(x-1)$ (Section~\ref{sec:shift});
\item a bar involution $\overline{\raisebox{0.5em}{\;\;}}\colon \mathcal{R}_q \to \mathcal{R}_q$ defined by $\overline{q} := q^{-1}$ and $\overline{x} := -qx$ (Section~\ref{sec:involution});
\item a Frobenius map $\Psi_p\colon \mathcal{R} \otimes_{\mathbb{Z}} \mathbb{F}_p \to \mathcal{R} \otimes_{\mathbb{Z}} \mathbb{F}_p$ defined by $\Psi_p(\binom{x}{k}) := \Psi_p(\binom{x}{pk})$ for all primes $p$ (Section~\ref{sec:frobenius});
\item a quantum Frobenius map $\Psi_d\colon \mathcal{R} \to \mathcal{R}_q/\Phi_d(q)$ defined by~\mbox{$\Psi_d(\binom{x}{k}) := \qbinom{x}{dk}$} for all integers~$d \geq 1$, where $\Phi_d$ is the $d$th cyclotomic polynomial (Section~\ref{sec:frobenius}).
\end{itemize}
And these maps have the following relations between them:
\begin{itemize}
\item $S\overline{P(x)} = \overline{S^{-1}P(x)}$ for all $P(x) \in \mathcal{R}_q$ (Proposition~\ref{prop:otherinvs});
\item $\Psi_p(S\overline{P(x)}) = S\overline{\Psi_p(P(x))}$ for all primes $p$, $P(x) \in \mathcal{R} \otimes_{\mathbb{Z}} \mathbb{F}_p$ (Proposition~\ref{prop:frobcommute});
\item $\Psi_d(S\overline{P(x)}) = S\overline{\Psi_d(P(x))}$ for all $d \geq 1$, $P(x) \in \mathcal{R}$ (Proposition~\ref{prop:qfrobcommute}).
\end{itemize}
We also have the following positivity properties for $\mathcal{R}_q$: for all~$i,j,m \in \mathbb{N}$,
\begin{itemize}
\item with $\qbinom{x}{i}\qbinom{x}{j} = \sum_{k} \alpha_{i,j,k}(q) \qbinom{x}{k}$, $\alpha_{i,j,k}(q) \in \mathbb{N}[q]$ for all $k$ (Theorem~\ref{thm:qbinommult});
\item with $\overline{\qbinom{x}{i}} \, \overline{ \qbinom{x}{j}} = \sum_{k} \overline{\alpha}_{i,j,k}(q)\overline{ \qbinom{x}{k}}$, $\overline{\alpha}_{i,j,k}(q) \in \mathbb{N}[q^{-1}]$ for all $k$ (Corollary~\ref{cor:negqbinom});
\item with $S^{m}\qbinom{x}{i} = \sum_{k} \beta_{m,i,k}(q) \qbinom{x}{k}$, $\beta_{m,i,k}(q) \in \mathbb{N}[q]$ for all $k$ (Equation~(\ref{eq:qbinommultishift}));
\item with $S^{-m}\overline{ \qbinom{x}{i}} = \sum_{k} \overline{\beta}_{m,i,k}(q) \overline{ \qbinom{x}{k}}$, $\overline{\beta}_{m,i,k}(q) \in \mathbb{N}[q^{-1}]$ for all $k$ (Equation~(\ref{eq:negqbinommultishift}));
\item with $\overline{ \qbinom{x}{i}} = \sum_{k} \gamma_{i,k}(q)\qbinom{x}{k}$, $\gamma_{i,k}(q) \in (-1)^i\mathbb{N}[q]$ for all $k$ (Proposition~\ref{prop:invconsts});
\item with $\qbinom{x}{i} = \sum_{k} \overline{\gamma}_{i,k}(q)\overline{ \qbinom{x}{k}}$, $\overline{\gamma}_{i,k}(q) \in (-1)^i\mathbb{N}[q^{-1}]$ for all $k$ (Proposition~\ref{prop:invconsts}).
\end{itemize}
Moreover, we offer simple, combinatorial formulas for all the coefficients above. Finally, using the tools we develop we classify all ring homomorphisms from~$\mathcal{R}_q$ into a field (Theorem~\ref{thm:mapstofields}). In general there is no reason to expect to be able to classify maps from a non-Noetherian commutative ring into a field, so this classification shows that indeed $\mathcal{R}_q$ has a very special structure. Especially important for this classification of maps from~$\mathcal{R}_q$ into a field is the aforementioned quantum Frobenius map. In turn, the construction of this quantum Frobenius map relies on a $q$-analog of Lucas' celebrated theorem~\cite{lucas1878theorie} due (we believe) to Sved~\cite{sved1988divisibility}. In the last section of the paper, Section~\ref{sec:future}, we discuss some open questions and future directions in the investigation of the ring $\mathcal{R}_q$.

In a recent paper by the first author \cite{harman2015stability} the ring $\mathcal{R}$ played an important role in understanding the asymptotic behavior of the modular representation theory of symmetric groups.  Part of the motivation for this paper was to understand what ring plays the role of $\mathcal{R}$ in the asymptotic behavior of Iwahori-Hecke algebras in type $A$. This direction will be addressed in more detail in an upcoming paper by the first author.

\noindent {\bf Acknowledgments}: We thank Fedor Petrov~\cite{petrov2015qint} for directing us to the work of Bhargava~\cite{bhargava1997p}, and for pointing out that the classical method of polynomial interpolation works to prove Proposition~\ref{prop:qbinom}.

\end{section}

\begin{section}{Quantum integer-valued polynomials} \label{sec:intvalpolys}
A polynomial $P(x) \in \mathbb{Q}[x]$ is \emph{integer-valued} if~$P(n) \in \mathbb{Z}$ for all $n \in \mathbb{N}$.  Let $\mathcal{R}$ denote the ring of such polynomials. We have the following proposition about the structure of $\mathcal{R}$, which in fact was essentially known to Newton.

\begin{proposition}[P\'olya 1919~\cite{polya1915ganzwertige}] \label{prop:binom}
$\mathcal{R}$ is freely generated as an abelian group by the binomial coefficient polynomials  $\binom{x}{k}$ for $k \in \mathbb{N}$ defined by
\[\binom{x}{k} := \frac{x(x-1)\dots(x-k+1)}{k!} \qquad \textrm{if $k \geq 1$},\]
with $\binom{x}{0} := 1$.
\end{proposition}

The key observation leading to this paper is that this ring admits a remarkable $q$-deformation called the ring of \emph{quantum integer-valued} polynomials. Recall the \emph{$q$-numbers} defined by $[n]_q := (q^n-1)/(q-1) =  (1 + q  + \dots +q^{n-1})$ for~$n \in \mathbb{N}$, with $[0]_q:=0$ by convention.  With these we may also define the \emph{$q$-factorials} $[n]_q! := [n]_q[n-1]_q\dots[1]_q$ for $n \in \mathbb{N}$, with the convention~$[0]_q! := 1$. For $n \in \mathbb{N}$ and $k \in \mathbb{Z}$ we define the \emph{$q$-binomial coefficients} by~$\qbinom{n}{k}_q := \frac{[n]_q!}{[n-k]_q! [k]_q!}$ when~$0 \leq k \leq n$, and~$\qbinom{n}{k}_q := 0$ if $k > n$ or~$k < 0$. Note the symmetry~$\qbinom{n}{k}_q = \qbinom{n}{n-k}_q$. Also note that~$\qbinom{n}{k}_q$ is a polynomial in $q$. In fact, $\qbinom{n}{k}_q \in \mathbb{N}[q]$, which follows from Lemma~\ref{lem:combint} below.

Now we define our main object of study, a $q$-deformation~$\mathcal{R}^{+}_q$ of $\mathcal{R}$:
\[ \mathcal{R}^{+}_q := \{P(x) \in \mathbb{Q}(q)[x]\colon P([n]_q) \in\mathbb{Z}[q] \textrm{ for all } n \in \mathbb{N}\}.\]
(The plus sign superscript will be explained in Section~\ref{sec:relrings} where we define a slightly larger ring $\mathcal{R}_q$ of which~$\mathcal{R}^{+}_q$ can be seen as the ``positive part.'') Note that $\mathcal{R}_{q}^{+}$ is naturally a $\mathbb{Z}[q]$-algebra. We have the following $q$-analog of Proposition~\ref{prop:binom}:

\begin{proposition}\label{prop:qbinom}
$\mathcal{R}^{+}_q$ is freely generated as a $\mathbb{Z}[q]$-module by the $q$-binomial coefficient polynomials $\qbinom{x}{k}$ for $k \in \mathbb{N}$ defined by
\[\qbinom{x}{k} := \frac{x(x-[1]_q)\dots(x-[k-1]_q)}{q^{\binom{k}{2}}[k]_q!} \qquad \textrm{if $k \geq 1$},\]
with $\qbinom{x}{0} := 1$. These polynomials satisfy $\qbinom{[n]_q}{k} = \qbinom{n}{k}_q$.
\end{proposition}

\begin{proof}
This proposition falls into a general framework set up by Bhargava; it can be seen as an instance of~\cite[Theorem 14]{bhargava1997p}. It also is essentially the same as~\cite[Chapter~II, Exercise~15]{cahen1997integer}, which in turn cites~\cite{gramain1990fonctions}. But let us give a self-contained proof based on a well-known proof of Proposition~\ref{prop:binom} using polynomial interpolation.

Verifying that when $x := [n]_q$ the $q$-binomial coefficient polynomials evaluate to the $q$-binomial coefficients $\qbinom{n}{k}_q$, and hence that these polynomials are actually in~$\mathcal{R}^{+}_q$, is a straightforward calculation.  Also, the $\qbinom{x}{k}$ are linearly independent just because of degree considerations. What remains is to check that everything in $\mathcal{R}^{+}_q$ is a $\mathbb{Z}[q]$-linear combination the $q$-binomial coefficient polynomials.

Let $P(x) \in \mathcal{R}^{+}_q$. We will construct polynomials $P_i(x)$ in the $\mathbb{Z}[q]$-span of the $q$-binomial coefficient polynomials for~$i\in \mathbb{N}$ such that $P_k([j]_q) = P([j]_q)$ for~$j\in[0,k]$ and such that~$P_i(x)$ has degree at most $k$.  The construction is given inductively as follows:
\begin{align*}
P_0(x) &:= P(0); \\
P_k(x) &:= P_{k-1}(x) + (P([k]_q) - P_{k-1}([k]_q))\qbinom{x}{k} \qquad \textrm{if $k \geq 1$}.
\end{align*}
By supposition~$P([k]_q)\in \mathbb{Z}[q]$ and so~$P([k]_q) - P_{k-1}([k]_q) \in \mathbb{Z}[q]$ for all $k \in \mathbb{N}$. Also, since~$\qbinom{[k]_q}{k} = \qbinom{k}{k}_q = 1$ and $\qbinom{[n]_q}{k} = \qbinom{n}{k}_q = 0$ if~$n < k$, for all $k\in\mathbb{N}$ we have that~$P(x)-P_k(x)$ vanishes at the points~$x=[0]_q,[1]_q,\ldots,[k]_q$. So if $P(x)$ has degree~$d$ then $P(x) - P_d(x)$ is a polynomial of degree at most~$d$ vanishing at the $(d+1)$ points $x= [0]_q, [1]_q, \dots [d]_q$ and is thus the zero polynomial. We conclude that~$P(x) = P_d(x)$ and so have successfully expressed $P(x)$ as a~$\mathbb{Z}[q]$-linear combination the $q$-binomial coefficient polynomials, as desired.
\end{proof}

\end{section}

\begin{section}{Combinatorial interpretations and polynomial interpolations}

In this section we will review some well-known combinatorial interpretations for the ($q$-)binomial coefficients. In order to do that we need to review some notation for partitions. Recall that a \emph{partition $\lambda = (\lambda_1,\lambda_2,\lambda_3,\ldots)$} is an infinite nonincreasing sequence of nonnegative integers that is eventually zero. We write $\lambda = (\lambda_1,\ldots,\lambda_k)$ to mean that $\lambda_j = 0$ for $j > k$. The \emph{size $|\lambda|$ of $\lambda$} is~$|\lambda| := \sum_{i=0}^{\infty} \lambda_i$. The \emph{length $\ell(\lambda)$ of $\lambda$} is~$\ell(\lambda) := \mathrm{min}\{i\colon i \in \mathbb{N}, \lambda_{i+1}=0\}$. There is a unique partition~$\lambda$ with~$|\lambda| = 0$ (which also has $\ell(\lambda) = 0$) called the \emph{empty partition}, and it is denoted by~$\varnothing$.  Associated to a partition $\lambda$ is its \emph{Young diagram}, which is the topleft-aligned collection of boxes having $\lambda_i$ boxes in row~$i$. For example, the Young diagram of $\lambda = (4,4,2,1)$ is: {\tiny \[\ydiagram{4,4,2,1}\]}Partitions are partially-ordered by containment of Young diagrams: for partitions $\lambda$ and $\mu$ we write $\lambda \subseteq \mu$ to mean $\lambda_i \leq \mu_i$ for all~$i$. The \emph{conjugate partition} of~$\lambda$, denoted $\lambda' = (\lambda'_1,\lambda'_2,\ldots)$, is the partition whose Young diagram is the transpose of the Young diagram of $\lambda$. Equivalently,~$|\{j\colon \lambda'_j = i\}| = \lambda_i - \lambda_{i+1}$ for all $i \geq 1$. Finally, for $m,k \in \mathbb{N}$ the \emph{rectangular partition $m^k$} is the partition~$m^k := (\overbrace{m,m,\ldots,m}^{k})$, which is $\varnothing$ if either $m$ or $k$ are equal to zero.

\begin{lemma} \label{lem:combint}
We have the following interpretations of~$\binom{n}{k}$ and~$\qbinom{n}{k}_q$ for $n,k \in \mathbb{N}$:
\begin{enumerate}
\item \textbf{(Classical)} $\binom{n}{k}$ is the number of $k$-element subsets of $\{1,2,\ldots,n\}$.

\item \textbf{(Quantum)} $\qbinom{n}{k}_q  = \sum_{\lambda \subseteq (n-k)^{k}} q^{|\lambda|}$, where this sum is $0$ if $k > n$.

\item \textbf{(Finite Field)} $\qbinom{n}{k}_q$ evaluated at a prime power $q := p^m$ is the number of $k$-dimensional subspaces of $\mathbb{F}_{p^m}^n$.
\end{enumerate}
\end{lemma}

Although Lemma~\ref{lem:combint} is very well-known (see for example~\cite[Propositions~1.7.2 and~1.7.3]{stanley1996ec1}), we include a (standard) proof for completeness and because some of the same ideas that go into the proof will reappear later, especially in the next section when we compute structure constants.

\begin{proof}[Proof of Lemma~\ref{lem:combint}]
Let $n,k \in \mathbb{N}$. We will assume $k \leq n$ as otherwise all quantities in question are zero.

First let us address the finite field statement. So $q := p^m$ is a prime power in this paragraph. Let $V \subseteq \mathbb{F}_{q}^{n}$ be a $k$-dimensional subspace. The orbit of $V$ under the action of the general linear group~$\mathrm{GL}(\mathbb{F}_{q}^{n})$ is exactly the set of all $k$-dimensional subspaces of~$\mathbb{F}_{q}^{n}$. Thus by the Orbit-Stabilizer Theorem we only need to compute~$|\mathrm{GL}(\mathbb{F}_{q}^{n})|$ and~$|\mathrm{Stab}(V)|$ to count the number of $k$-dimensional subspaces of $\mathbb{F}_{q}^{n}$. To compute the order of $\mathrm{GL}(\mathbb{F}_{q}^{n})$ let us represent elements of~$\mathrm{GL}(\mathbb{F}_{q}^{n})$ as matrices. Suppose we build a matrix in~$\mathrm{GL}(\mathbb{F}_{q}^{n})$ one row at a time. For the first row we have $(q^{n}-1)$ choices as any nonzero vector is permissible; for the second row we need to choose a vector not in the span of the first row and so have $(q^{n}-q)$ choices; for the third row we need to choose a vector not in the span of the first two rows and so have $(q^{n}-q^2)$ choices; and so on. Thus
\begin{equation} \label{eq:gln}
|\mathrm{GL}(\mathbb{F}_{q}^{n})| = (q^{n}-1)(q^{n}-q)\cdots(q^{n}-q^{n-1}).
\end{equation}
Computing $|\mathrm{Stab}(V)|$ is similar. Let us suppose without loss of generality that~$V$ is the span of the first $k$ standard basis vectors. Then $\mathrm{Stab}(V)$ consists of all elements of $\mathrm{GL}(\mathbb{F}_{q}^{n})$ of the form $\begin{pmatrix} A & 0 \\ * & * \end{pmatrix}$ where $A \in \mathrm{GL}(\mathbb{F}_{q}^{k})$. By~(\ref{eq:gln}), the number of choices for $A$ is $|\mathrm{GL}(\mathbb{F}_{q}^{k})| = (q^{k}-1)(q^{k}-q)\cdots(q^{k}-q^{k-1})$. The number of choices for the remaining $(n-k)$ rows of the matrix can be computed as follows: we have $(q^{n}-q^{k})$ choices for the $(k+1)$st row as it cannot lie in the span of the first~$k$ rows; we have $(q^{n}-q^{k+1})$ choices for the $(k+2)$nd row; and so on. So,
\[ |\mathrm{Stab}(V)| = (q^{k}-1)(q^{k} - q)\cdots(q^{k}-q^{k-1}) \cdot (q^{n}-q^{k})(q^{n}-q^{k+1})\cdots (q^{n}-q^{n-1}).\]
Therefore, by the Orbit-Stabilizer Theorem the number of $k$-dimensional subspaces of~$\mathbb{F}_{q}^{n}$ is
\[\frac{|\mathrm{GL}(\mathbb{F}_{q}^{n})|}{|\mathrm{Stab}(V)|} = \frac{[n]_q!\,(q-1)^{n}\,q^{\binom{n}{2}}}{[k]_q!\,(q-1)^{k}\,q^{\binom{k}{2}}\,[n-k]_q!\,(q-1)^{n-k}\,q^{\binom{n}{2}-\binom{k}{2}}} = \qbinom{n}{k}_q, \]
as claimed.

Now let us address the quantum statement. It is routine to verify that the $q$-binomial coefficients satisfy the $q$-Pascal identity for $n,k \in \mathbb{N}$:
\begin{equation} \label{eq:qpascal}
\qbinom{n}{k}_q = q^k\qbinom{n-1}{k}_q + \qbinom{n-1}{k-1}_q.
\end{equation}
Clearly if $k = 0$ or $n-k = 0$ we have~$\sum_{\lambda \subseteq (n-k)^k} q^{|\lambda|} = 1 = \qbinom{n}{k}_q$. So assume that~$k > 0$ and $n-k > 0$. To establish that $\sum_{\lambda \subseteq (n-k)^k} q^{|\lambda|} = \qbinom{n}{k}_q$ we need only show that $\sum_{\lambda \subseteq (n-k)^k} q^{|\lambda|}$ satisfies the same recurrence as in~(\ref{eq:qpascal}). We can establish this recurrence bijectively: specifically, we define a bijection 
\begin{align*}
\{\lambda\colon \lambda \subseteq (n-k)^k\} &\xrightarrow{\sim} \{(i,\mu)\colon i \in \{0,1\}, \mu \subseteq (n-1-k+i)^{k-i}\} \\
\lambda &\mapsto \begin{cases} (0,(\lambda_1-1,\lambda_2-1,\ldots,\lambda_k-1)) &\textrm{if $\ell(\lambda)=k$}; \\
(1,(\lambda_1,\lambda_2,\ldots,\lambda_{k-1})) &\textrm{if $\ell(\lambda) < k$}, \end{cases}
\end{align*}
and observe that if $\lambda \mapsto (i,\mu)$ under this bijection then $|\lambda| = k(1-i) + |\mu|$. So the quantum statement is proved. We remark that the quantum and finite field interpretations of the $q$-binomial coefficients are closely connected via the Schubert cell decomposition of the Grassmannian (and indeed this decomposition is an alternative way to prove the quantum statement; see the proof of~\cite[Proposition~1.7.3]{stanley1996ec1}).

Finally, let us show how that classical statement follows from the quantum one. The point is that there is a bijection between partitions contained in~$(n-k)^k$ and $k$-element subsets of~$\{1,\ldots,n\}$. The bijection works as follows. First we represent a partition $\lambda \subseteq (n-k)^k$ by the southeast border path of its Young diagram, which connects the southwest corner of the rectangle $(n-k)^k$ to its northeast corner. This southeast border path consists of exactly~$k$ north steps and~$(n-k)$ east steps. We map $\lambda$ to the set $S \subseteq \{1,\ldots,n\}$ of the indices of the north steps in this path. For example, if $n=7$, $k=3$, and $\lambda = (4,2,0)$, then the following depicts the Young diagram of $\lambda$ inside $(n-k)^k$ (its boxes are shaded) together with its southeast border path (in bold) with the steps of this path labeled by their indices (the labels of north steps are to the right of the step and the labels of east steps are above the step):
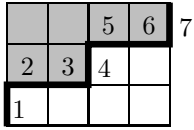
\begin{center}
\begin{tikzpicture}
\node at (0,0) {\begin{ytableau} *(lightgray) & *(lightgray) & *(lightgray) & *(lightgray)  \\ *(lightgray) & *(lightgray)  & &  \\ & & &  \end{ytableau}};
\def\x{0.54}
\draw[line width=2.5pt] (-2.0*\x,-1.525*\x) -- (-2.0*\x,-0.5*\x) -- (0*\x,-0.5*\x) -- (0*\x,0.5*\x) -- (2.0*\x,0.5*\x) -- (2.0*\x,1.525*\x);
\node at (-1.7*\x,-1.1*\x) {1};
\node at (-1.5*\x,-0.05*\x) {2};
\node at (-0.5*\x,-0.05*\x) {3};
\node at (0.4*\x,-0.1*\x) {4};
\node at (0.5*\x,0.95*\x) {5};
\node at (1.5*\x,0.95*\x) {6};
\node at (2.4*\x,0.9*\x) {7};
\end{tikzpicture}
\end{center}
So $\lambda$ is sent to $\{1,4,7\}$. This bijection between partitions and subsets establishes the identity
\begin{equation} \label{eq:qbinomsubsets}
\qbinom{n}{k}_q = \sum_{\lambda \subseteq (n-k)^k} q^{|\lambda|} = \sum_{S \subseteq \{1,2,\ldots,n\}} q^{\sum_{i \in S} i - \binom{k+1}{2}}.
\end{equation}
The classical statement follows from~(\ref{eq:qbinomsubsets}) by specializing $q :=1$.
\end{proof}

We will be using the interpretations in Lemma~\ref{lem:combint} to give combinatorial proofs of algebraic identities within these rings. In order to pass from the combinatorial interpretations to algebraic expressions involving $x$ we have the following easy lemma.

\begin{lemma} \label{lem:interpolation} \

\begin{enumerate}

\item \textbf{(Classical interpolation)} Let $F \in \mathbb{Z}[x_0,x_1,...x_k]$. Then 
\[F\left(\binom{n}{0}, \binom{n}{1}, \dots , \binom{n}{k}\right) = 0 \textrm{ in $\mathbb{Z}$}\] 
for all $n\in \mathbb{N}$ iff $F( \binom{x}{0}, \binom{x}{1}, \dots , \binom{x}{k}) = 0 $ in $\mathcal{R}$. 

\item {\textbf{(Quantum interpolation)}} Let $F \in \mathbb{Z}[q][x_0,x_1,...x_k]$. Then 
\[F\left(\qbinom{n}{0}_q, \qbinom{n}{1}_q, \dots , \qbinom{n}{k}_q\right) = 0 \textrm{ in $\mathbb{Z}[q]$}\]
for all $n\in \mathbb{N}$ iff $F(\qbinom{x}{0}, \qbinom{x}{1}, \dots , \qbinom{x}{k}) = 0 $ in $\mathcal{R}^{+}_q$. 

\item {\textbf{(Finite field interpolation)}} Let $F \in \mathbb{Z}[q][x_0,x_1,...x_k]$. Then
\[F\left(\qbinom{n}{0}_{q}, \qbinom{n}{1}_{q}, \dots , \qbinom{n}{k}_{q}\right) = 0 \textrm{ in $\mathbb{Z}[q]/\langle q-p^m\rangle$}\] 
for all $n\in \mathbb{N}$ and prime powers $p^m$ iff $F(\qbinom{x}{0}, \qbinom{x}{1}, \dots , \qbinom{x}{k}) = 0 $ in $\mathcal{R}^{+}_q$. 

\end{enumerate}
\end{lemma}

\begin{proof}
The first two statements just reduce to the fact that a polynomial in one variable vanishes at infinitely many points if and only if it is the zero polynomial.  For the third statement we can use the fact mentioned in the previous sentence to go from $q := p^m$ a prime power to formal $q$, and then apply the second statement.
\end{proof}

\end{section}

\begin{section}{Structure constants} \label{sec:structconsts}

Propositions~\ref{prop:binom} and~\ref{prop:qbinom} give us bases for $\mathcal{R}$ and $\mathcal{R}^{+}_q$ as algebras over~$\mathbb{Z}$ and~$\mathbb{Z}[q]$ respectively, but they do not tell us anything about how a product of basis elements decomposes as a sum of other basis elements. The point of this section will be to give formulas and combinatorial interpretations for these structure constants.

While $\mathcal{R}$ can be obtained from $\mathcal{R}^{+}_q$ by specializing $q := 1$, it is an important specialization so we will treat it on its own as a warm up. The following theorem gives the structure constants for $\mathcal{R}$.

\begin{theorem}\label{thm:binommult} 
Let $i,j \in \mathbb{N}$. Then the following formula holds in $\mathcal{R}$:
\[\binom{x}{i}\binom{x}{j} = \sum_{k = \mathrm{max}(i,j)}^{i+j} \frac{k!}{(k-i)!(k-j)! (i+j-k)!} \binom{x}{k}.\]
In particular, the $\mathbb{Z}$-algebra $\mathcal{R}$ with distinguished basis $\{\binom{x}{k}\colon k \in \mathbb{N}\}$ has structure constants in~$\mathbb{N}$.
\end{theorem}

\begin{proof} 
 Note that for $k \in [\mathrm{max}(i,j),i+j]$ we have
\[\frac{k!}{(k-i)!(k-j)! (i+j-k)!} = \binom{k}{i+j-k} \binom{2k-i-j}{k-i},\]
an equality which also explains the comment about the structure constants belonging to~$\mathbb{N}$. So by Lemma~\ref{lem:interpolation} it suffices to verify
\begin{equation} \label{eq:binommult}
\binom{n}{i}\binom{n}{j} = \sum_{k = \mathrm{max}(i,j)}^{i+j}  \binom{n}{k} \binom{k}{i+j-k} \binom{2k-i-j}{k-i}
\end{equation}
for $n \in \mathbb{N}$. For this we claim that both sides of~(\ref{eq:binommult}) count the same thing according to the classical interpretation of the binomial coefficients stated in Lemma~\ref{lem:combint}: namely, the number of pairs of subsets~$I, J \subseteq \{1,2,\dots,n\}$ such that~$|I|=i$ and~$|J|=j$. For the left-hand side of~(\ref{eq:binommult}) this is obvious. The right-hand side of~(\ref{eq:binommult}) can be interpreted as follows: the term $\binom{n}{k}$ counts the number of ways to choose~$K := I \cup J \subseteq \{1,\ldots,n\}$ of size $k := |K| \in [\mathrm{max}(i,j),i+j]$; the term $\binom{k}{i+j-k}$ counts the number of ways to choose the subset $I \cap J$ inside of~$K$; and the term $\binom{2k-i-j}{k-i}$ counts the number of ways to divide $K\setminus (I\cap J)$ into $J \setminus I$ (which has~$|J\setminus I|=k-i$) and~$I \setminus J$ (which has~$|I \setminus J| = k-j$). But these choices are equivalent to just choosing the subsets~$I$ and~$J$.
\end{proof}

We can mimic this entire theorem and proof (using subspaces of finite vector spaces instead of subsets of finite sets) to obtain the structure constants for~$\mathcal{R}^{+}_q$.

\begin{theorem} \label{thm:qbinommult} 
Let $i,j \in \mathbb{N}$. Then the following formula holds in $\mathcal{R}^{+}_q$:
\[ \qbinom{x}{i} \qbinom{x}{j} = \sum_{k = \mathrm{max}(i,j)}^{i+j}  \frac{q^{(k-i)(k-j)} \, [k]_q!}{[k-i]_q![k-j]_q![i+j-k]_q!} \qbinom{x}{k}.\]
In particular, the $\mathbb{Z}[q]$-algebra $\mathcal{R}^{+}_q$ with distinguished basis $\{\qbinom{x}{k}\colon k \in \mathbb{N}\}$ has structure constants in~$\mathbb{N}[q]$.
\end{theorem}
\begin{proof}[Proof of Theorem~\ref{thm:qbinommult} via finite Grassmannians]
First let us prove an auxilliary result. Let $V$ be a $v$-dimensional vector space over $\mathbb{F}_{p^m}$. Then we claim that the number of pairs of subpsaces $U, W \subseteq V$ where $U$ is $u$-dimensional, $W$ is~$(v-u)$-dimensional, $U \cap W = 0$, and $U + W = V$ is $q^{u(v-u)}\qbinom{v}{u}_q$, where $q := p^m$. To prove this claim we will use the Orbit-Stabilizer Theorem as in the proof of Lemma~\ref{lem:combint}. Note that $\mathrm{GL}(V)$ acts on such pairs $U,W$ by simultaneously acting on~$U$ and on~$W$. There is one $\mathrm{GL}(V)$-orbit consisting of all such pairs. By~(\ref{eq:gln}) we have that $|\mathrm{GL}(V)| = (q^{v}-1)\cdots(q^{v}-q^{v-1})$. Also, for any fixed pair~$U,W$ it is clear that $\mathrm{Stab}(U,W) = \mathrm{GL}(U) \times \mathrm{GL}(W)$. So by~(\ref{eq:gln}) we have that $|\mathrm{Stab}(U,W)| = (q^{u}-1)\cdots(q^{u}-q^{u-1})\cdot(q^{v-u}-1)\cdots(q^{v-u}-q^{v-u-1})$. Thus by the Orbit-Stabilizer Theorem the number of such pairs $U,W$ is 
\[\frac{|\mathrm{GL}(V)|}{|\mathrm{Stab}(U,W)|} = \frac{[v]_q!\,(q-1)^v\,q^{\binom{v}{2}}}{[u]_q!\,(q-1)^u\,q^{\binom{u}{2}}\,[v-u]_q!\,(q-1)^{v-u}\,q^{\binom{v-u}{2}}} = q^{u(v-u)}\qbinom{v}{u}_q\]
as claimed.

Now we return to the proof of the theorem. Note that for~$k \in [\mathrm{max}(i,j),i+j]$ we have
\[\frac{q^{(k-i)(k-j)} \, [k]_q!}{[k-i]_q![k-j]_q![i+j-k]_q!} = q^{(k-i)(k-j)} \qbinom{k}{i+j-k}_q \qbinom{2k-i-j}{k-i}_q,\]
an equality which also explains the comment about the structure constants belonging to~$\mathbb{N}[q]$. So by Lemma~\ref{lem:interpolation} it suffices to verify
\begin{equation} \label{eq:qbinommult}
\qbinom{n}{i}_q \qbinom{n}{j}_q = \sum_{k = \mathrm{max}(i,j)}^{i+j}  \qbinom{n}{k}_q \qbinom{k}{i+j-k}_q q^{(k-i)(k-j)} \qbinom{2k-i-j}{k-i}_q
\end{equation}
for all $n \in \mathbb{N}$ and $q := p^m$ a prime power. For this we claim that both sides of~(\ref{eq:qbinommult}) count the same thing according to the finite field interpretation of the $q$-binomial coefficients stated in Lemma~\ref{lem:combint}: namely, the number of pairs of subspaces~$I, J \subseteq \mathbb{F}_{p^m}^n$ such that $I$ is $i$-dimensional and $J$ is $j$-dimensional. For the left-hand side of~(\ref{eq:qbinommult}) this is obvious. The right-hand side of~(\ref{eq:qbinommult}) can be interpreted as follows: the term $\qbinom{n}{k}_q$ counts the number of ways to choose a subspace $K := I+J$ of $\mathbb{F}_{p^m}^{n}$ of dimension $k := \mathrm{dim}(K) \in [\mathrm{max}(i,j),i+j]$; the term $\qbinom{k}{i+j-k}_q$ counts the number of ways to choose the subspace $I \cap J$ of dimension $\mathrm{dim}(I \cap J) = i+j-k$ inside $K$; and, thanks to the first paragraph, the term $q^{(k-i)(k-j)}\qbinom{2k-i-j}{k-i}_q$ counts the number of ways to split $K/(I\cap J)$ into~$\pi(I),\pi(J)$, where $\pi\colon K \to K/(I \cap J)$ is the projection map. But these choices are equivalent to just choosing the subspaces~$I$ and~$J$.
\end{proof}

Lemma~\ref{lem:combint} gives us two different interpretations of $q$-binomial coefficients. It is always worthwhile to interpret an identity of $q$-binomial coefficients like the one in Theorem~\ref{thm:qbinommult} in both the language of finite Grassmannians and of Young diagram combinatorics. Thus we will now give a different, bijective proof of Theorem~\ref{thm:qbinommult} using Young diagrams.

\begin{proof}[Proof of Theorem~\ref{thm:qbinommult} via Young diagrams] Assume by symmetry that~$i \geq j$. Then note that for $k \in [i,i+j]$ we have
\[\frac{q^{(k-i)(k-j)} \, [k]_q!}{[k-i]_q![k-j]_q![i+j-k]_q!} = q^{(k-i)(k-j)} \qbinom{k}{i}_q \qbinom{i}{k-j}_q. \]
So by Lemma~\ref{lem:interpolation} it suffices to verify
\begin{equation} \label{eq:qbinommult2}
\qbinom{n}{i}_q \qbinom{n}{j}_q = \sum_{k = i}^{i+j} q^{(k-i)(k-j)} \qbinom{n}{k}_q \qbinom{k}{i}_q \qbinom{i}{k-j}_q 
\end{equation}
for all~$n \in \mathbb{N}$. Equation~(\ref{eq:qbinommult2}) will follow from the existence of a bijection
\[\arraycolsep=1pt \left\{(\lambda,\mu)\colon \begin{array}{c} \lambda \subseteq (n-i)^i, \\ \mu \subseteq (n-j)^j \end{array}\right\} \xrightarrow{\sim} \left\{(k,\alpha,\beta,\gamma)\colon \begin{array}{c} k \in [i,i+j], \alpha \subseteq (n-k)^k, \\ \beta \subseteq (k-i)^i, \gamma \subseteq (i+j-k)^{k-j} \end{array}\right\} \]
satisfying $|\lambda| + |\mu| = (k-i)(k-j) + |\alpha| + |\beta| + |\gamma|$ when $(\lambda,\mu) \mapsto (k,\alpha,\beta,\gamma)$. Indeed, Lemma~\ref{lem:combint} tells us that the left-hand side of~(\ref{eq:qbinommult2}) is $\sum_{(\lambda,\mu)}q^{|\lambda| + |\mu|}$ where the sum is over the domain of this bijection, and that the right-hand side is~$\sum_{(k,\alpha,\beta,\gamma)}q^{(k-i)(k-j) + |\alpha| + |\beta| + |\gamma|}$ where the sum is over the codomain. So let~$\lambda \subseteq (n-i)^i$ and~$\mu \subseteq (n-j)^j$; we will define $k$, $\alpha$, $\beta$, $\gamma$ such that setting~$(\lambda,\mu) \mapsto (k,\alpha,\beta,\gamma)$ gives us the desired bijection. First: how do we find the number~$k$ from $\lambda$ and~$\mu$? We define $k$ as follows:
\[k := \mathrm{max}\{m\in[i,i+j]\colon(m-j)^{m-i} \subseteq \mu\}. \]
This~$(k-j)^{k-i}$ rectangle in the topleft corner of $\mu$ will be removed from $\mu$ as we construct $\alpha$, $\beta$, $\gamma$ from the remaining boxes of~$\lambda$ and~$\mu$ and will account for the~$(k-i)(k-j)$ term in the desired equality $|\lambda| + |\mu| = (k-i)(k-j) + |\alpha| + |\beta| + |\gamma|$. There are some boxes in~$\mu$ south of the~$(k-j)^{k-i}$ rectangle and some boxes east. The boxes in~$\mu$ south of the rectangle will become $\gamma$; or rather, they will become the transpose of $\gamma$. That is, we set
\[\gamma := (\mu_{k-i+1},\mu_{k-i+2},\ldots,\mu_{j})'.\]
The boxes in~$\mu$ east of the rectangle will mix with boxes in $\lambda$ to form $\alpha$. More specifically, we will pull certain columns off of $\lambda$ to form~$\beta$, and the remaining boxes in~$\lambda$ will mix with the boxes in~$\mu$ to the east of the $(k-i)^{k-j}$ rectangle to form~$\alpha$. These columns of $\lambda$ are defined as follows: for~$s = 1,\ldots,k-i$ we set
\[c_s:= \mathrm{min}\{t\in \mathbb{N}\colon \lambda_{t+1} \leq \mu_s - (i-j+s)\}. \]
Note that $\mu_s \geq k-j \geq (i-j+s)$ for all $s = 1,\ldots,k-i$, and $\lambda_{i+1} = 0$, so~$c_s \in [0,i]$ for all $s =1,\ldots,k-i$. As we said, these columns become $\beta$; that is, we set
\[\beta := (c_{k-i},c_{k-i-1},\ldots,c_{1})'.\]
Finally, set
\[\alpha := \begin{array}{c}(\lambda_1 - (k-i),\lambda_2-(k-i),\lambda_3-(k-i),\ldots,\lambda_{c_1}-(k-i),\mu_1-(k-j), \\ \lambda_{c_1+1} - (k-i-1),\lambda_{c_1+2} - (k-i-1),...,\lambda_{c_2}-(k-i-1),\mu_2-(k-j),\\
\lambda_{c_2+1} - (k-i-2),\lambda_{c_2+2} - (k-i-2),...,\lambda_{c_3}-(k-i-2),\mu_3-(k-j),\\
\vdots \\
\lambda_{c_{k-i-1}+1} - 1,\lambda_{c_{k-i-1}+2} - 1,...,\lambda_{c_{k-i}}-1,\mu_{k-i}-(k-j), \\
\lambda_{c_{k-i}+1},\lambda_{c_{k-i}+2},\ldots,\lambda_{i}).
\end{array}\]

Let us illustrate this construction with an example. Say $n = 14$, $i=7$, $j = 6$, and $\lambda = (7,6,5,5,2,2,1)$ and $\mu = (8,7,6,4,2,1)$. Then we can compute~$k = 10$. Figure~\ref{fig:bijex} shows how the boxes of~$\lambda$ and the boxes of~$\mu$ not in the top left~$4^3$ rectangle are moved around to construct $\alpha$, $\beta$, and $\gamma$. We see that:
\[\alpha=(4,4,4,3,3,3,2,2,2,1); \qquad \beta=(3,2,2,2); \qquad \gamma=(3,2,1,1).\]
In this example the $c$'s are $c_1 = 1$, $c_2=4$, and~$c_3=4$. 

\begin{figure}[ht]
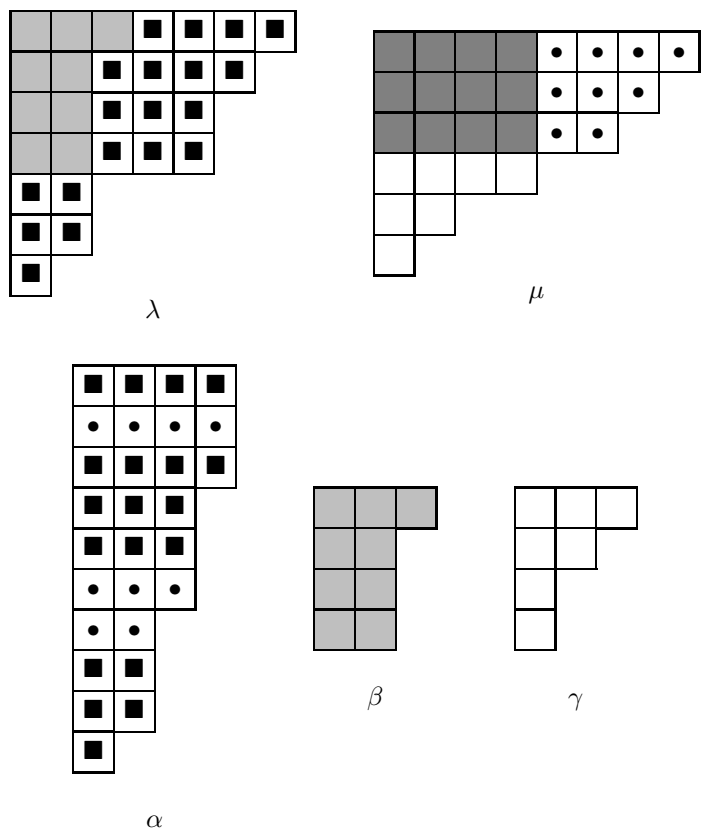

\[\begin{array}{c} \begin{ytableau} 
*(lightgray) & *(lightgray) & *(lightgray) &  \blacksquare  &  \blacksquare  &  \blacksquare  &  \blacksquare  \\ 
*(lightgray) & *(lightgray) &  \blacksquare  &  \blacksquare  &  \blacksquare  &  \blacksquare   \\ 
*(lightgray) & *(lightgray) &  \blacksquare  &  \blacksquare  &  \blacksquare   \\ 
*(lightgray) & *(lightgray) &  \blacksquare  & \blacksquare  &  \blacksquare  \\
\blacksquare  &  \blacksquare  \\
\blacksquare  &  \blacksquare  \\
\blacksquare  \\
\end{ytableau} \\ \lambda \end{array} \qquad
\begin{array}{c} \begin{ytableau} 
*(gray) & *(gray) & *(gray) & *(gray) & \bullet & \bullet & \bullet & \bullet  \\ 
*(gray) & *(gray) & *(gray) & *(gray) & \bullet & \bullet & \bullet  \\ 
*(gray) & *(gray) & *(gray) & *(gray) & \bullet & \bullet  \\ 
*(white) & *(white) & *(white) & *(white) \\
*(white) & *(white) \\
*(white) \\
\end{ytableau} \\ \mu \end{array}\] \medskip
\[ \begin{array}{c} \begin{ytableau} \blacksquare  & \blacksquare  & \blacksquare  & \blacksquare  \\
\bullet & \bullet & \bullet & \bullet \\ 
\blacksquare  & \blacksquare  & \blacksquare  & \blacksquare  \\
\blacksquare  & \blacksquare  & \blacksquare    \\ 
\blacksquare  & \blacksquare  & \blacksquare  \\
\bullet & \bullet & \bullet  \\ 
\bullet & \bullet  \\ 
\blacksquare  & \blacksquare  \\
\blacksquare  & \blacksquare  \\
\blacksquare  \\
\end{ytableau} \\ \\ \alpha \end{array} \qquad
\begin{array}{c} \begin{ytableau}
*(lightgray) & *(lightgray) & *(lightgray) \\
*(lightgray) & *(lightgray) \\
*(lightgray) & *(lightgray) \\
*(lightgray) & *(lightgray) \\
\end{ytableau} \\ \\ \beta \end{array} \qquad
\begin{array}{c} \begin{ytableau}
*(white) & *(white) & *(white) \\
*(white) & *(white) \\
*(white) \\
*(white) \\
\end{ytableau} \\ \\ \gamma \end{array}\]
\caption{An example of the bijection in the second proof of Theorem~\ref{thm:qbinommult}. In this example $n=14$, $i=7$, $j=6$, and $k=10$.} \label{fig:bijex}
\end{figure}

By definition $k \in [i,i+j]$. It is routine to verify that $\alpha \subseteq (n-k)^k$, $\beta \subseteq (k-i)^i$, and $\gamma \subseteq (i+j-k)^{k-j}$, and that $|\lambda| + |\mu| = (k-i)(k-j) + |\alpha| + |\beta| + |\gamma|$. It is also easy to check that this procedure is invertible. So we have indeed defined the desired bijection.
\end{proof}

Unlike the proof of Theorem~\ref{thm:binommult} via subsets or the proof of Theorem~\ref{thm:qbinommult} via subspaces, this last proof of Theorem~\ref{thm:qbinommult} via Young diagrams breaks the symmetry between~$i$ and~$j$. It would be interesting to find a bijective proof of Theorem~\ref{thm:qbinommult} that respects the symmetry between~$i$ and~$j$. In particular, it does not appear possible to directly transfer the bijective proof of Theorem~\ref{thm:binommult} to a bijective proof of Theorem~\ref{thm:qbinommult} via the correspondence between subsets and Young diagrams used to establish~(\ref{eq:qbinomsubsets}).

\end{section}

\begin{section}{Related rings} \label{sec:relrings}
Recall that the ring of integer-valued polynomials $\mathcal{R}$ was defined in Section~\ref{sec:intvalpolys} as the collection of all $P \in \mathbb{Q}[x]$ such that $P(n) \in \mathbb{Z}$ for all $n \in \mathbb{N}$. This was the definition we $q$-deformed to define $\mathcal{R}_q^+$. However, there are (at least) two other well known characterizations of $\mathcal{R}$, as summarized by the following lemma.

\begin{lemma}\label{lem:otherchars}
Let $P\in\mathbb{Q}[x]$ be a polynomial of degree $d$. Then the following are equivalent:
\begin{enumerate}
\item $P(n) \in \mathbb{Z}$ for $n\in[0,d]$.

\item $P(n) \in \mathbb{Z}$ for $n \in \mathbb{N}$ (in other words, $P \in \mathcal{R}$).

\item $P(n) \in \mathbb{Z}$ for $n \in \mathbb{Z}$.

\end{enumerate}
\end{lemma}

\begin{proof}
Clearly $3 \implies 2 \implies 1$, so we just need to check that~$1 \implies 3$. The proof will be by induction on the degree $d$.  Note that if~$d=0$ then the polynomial is constant and the theorem holds trivially. So assume the degree of~$P(x)$ is greater than zero and define~$\widetilde{P}(x):= P(x+1)-P(x)$ and note that~$\widetilde{P}$ is a degree $d-1$ polynomial in $\mathbb{Q}[x]$ taking integer values at~$x\in[0,d-1]$. Hence by induction $\widetilde{P}(n)\in \mathbb{Z}$ for all $n \in \mathbb{Z}$. Finally we can conclude~$P(n) \in \mathbb{Z}$ for all~$n \in \mathbb{Z}$ because
\[P(n)= \begin{cases} 
      P(0)+\tilde{P}(0)+\tilde{P}(1)+\dots+\tilde{P}(n-1) & n\geq 0, \\
      P(0)-\tilde{P}(-1)-\tilde{P}(-2)-\dots-\tilde{P}(n) & n< 0; 
   \end{cases}\]
and in either case, each term on the right hand side is an integer.
\end{proof}

We'd now like to have a $q$-analog of Lemma~\ref{lem:otherchars} for our ring $\mathcal{R}_q^+$. We'll note that the definition $[n]_q := (q^n-1)/(q-1) \in \mathbb{Z}[q,q^{-1}]$ makes perfect sense for all integers~$n \in \mathbb{Z}$. Explicitly, for $n \in \mathbb{Z}$ we set
\[[n]_q := \begin{cases} 1+q+q^2 + \cdots + q^{n-1} &\textrm{if $n > 0$};\\
0 &\textrm{if $n =0$};\\
-q^{-1}-q^{-2}-\cdots-q^{n} &\textrm{if $n < 0$}. \end{cases}\]
With these extended $q$-numbers in mind, we have the following $q$-analog of Lemma \ref{lem:otherchars}:

\begin{lemma}\label{lem:qotherchars}
Let $P\in\mathbb{Q}(q)[x]$ be a polynomial of degree $d$ with~$P([n]_q) \in \mathbb{Z}[q]$ for $n\in[0,d]$. Then:
\begin{enumerate}

\item $P([n]_q) \in \mathbb{Z}[q]$ for $n \in \mathbb{N}$ (in other words, $P \in \mathcal{R}_q^+$).

\item $P([n]_q) \in \mathbb{Z}[q,q^{-1}]$ for $n \in \mathbb{Z}$.

\end{enumerate}
\end{lemma}

\begin{proof}
We will follow the previous proof pretty closely, inducting on the degree~$d$ with the $d=0$ case holding trivially. So we assume that $P(x)$ has degree greater than zero, define~$\widetilde{P}(x):= P(qx+1)-q^dP(x)$ and again note that~$\widetilde{P}$ is a degree~$d-1$ polynomial in $\mathbb{Q}(q)[x]$ with $\widetilde{P}([n]_q) \in \mathbb{Z}[q]$ for all~$n\in[0,d-1]$. Hence by induction $\widetilde{P}(x)$ satisfies the conclusions of the lemma. Finally we conclude the statement for $P$ since
\[ P([n]_q)= \begin{cases} 
      q^{nd}P([0]_q)+\sum_{i=0}^{n-1}q^{(n-1-i)d}\widetilde{P}([i]_q) & n\geq 0, \\
      q^{nd}P([0]_q)-\sum_{i=1}^{-n}q^{(n+i)d}\widetilde{P}([-i]_q) & n< 0;
   \end{cases}
\]
and by induction we know each term on the right hand side is in either $Z[q]$ or~$Z[q,q^{-1}]$ as required.
\end{proof}

The second part of the previous lemma suggests that if we want a $q$-analog of~$\mathcal{R}$ which is symmetric for the positive and negative $q$-integers, then we should have the target be $\mathbb{Z}[q,q^{-1}]$ rather than $\mathbb{Z}[q]$.  This motivates the following definition:
\[\mathcal{R}_q := \{P(x) \in \mathbb{Q}(q)[x]\colon P([n]_q) \in\mathbb{Z}[q,q^{-1}] \textrm{ for all } n \in \mathbb{Z}\}.\] 
This is clearly a $\mathbb{Z}[q,q^{-1}]$-algebra, which by the previous lemma contains $\mathcal{R}_q^+$ as a $\mathbb{Z}[q]$-subalgebra. The following proposition says there is essentially nothing else in $\mathcal{R}_q$.

\begin{proposition} \label{prop:localization} 
$\mathcal{R}_q = \mathcal{R}_{q}^{+} \otimes_{\mathbb{Z}[q]} \mathbb{Z}[q,q^{-1}]$ viewed as subrings of $\mathbb{Q}(q)[x]$.
\end{proposition}

\begin{proof} 
As mentioned above, the inclusion $\mathcal{R}_{q}^{+} \otimes_{\mathbb{Z}[q]} \mathbb{Z}[q,q^{-1}] \subseteq \mathcal{R}_{q}$ is immediate from the second part of Lemma $\ref{lem:qotherchars}$. To see the other direction, let~$P(x)\in \mathcal{R}_q$ be of degree $d$.  There is a positive integer~$m$ such that~$q^mP([0]_q)$, $q^mP([1]_q)$, $\ldots$, $q^mP([d]_q)$ are all in $\mathbb{Z}[q]$. By the first part of Lemma~\ref{lem:qotherchars} it follows that the polynomial $q^mP(x)$ is in $\mathcal{R}_q^+$. Hence we have~$P(x) \in \mathcal{R}_{q}^{+} \otimes_{\mathbb{Z}[q]} \mathbb{Z}[q,q^{-1}]$ as desired.
\end{proof}

Proposition~\ref{prop:localization} means that many results we have proved about $\mathcal{R}_q^{+}$ transfer directly to $\mathcal{R}$: for example, the $q$-binomial coefficient polynomials $\qbinom{x}{k}$ are a~$\mathbb{Z}[q,q^{-1}]$-basis of~$\mathcal{R}$ and their structure constants are still as in Theorem~\ref{thm:qbinommult}. 

There is an obvious counterpart to~$\mathcal{R}_q^{+}$ where we plug in negative $q$-numbers instead of positive ones. We define this ``negative'' part $\mathcal{R}_q^{-}$ of $\mathcal{R}_q$ as follows:
\[\mathcal{R}_q^{-} := \{P(x) \in \mathbb{Q}(q)[x]\colon P([-n]_q) \in\mathbb{Z}[q^{-1}] \textrm{ for all } n \in \mathbb{N}\}.\]
Note that $\mathcal{R}_q^{-}$ is a $\mathbb{Z}[q^{-1}]$-subalgebra of $\mathcal{R}_q$. We will see in Section~\ref{sec:involution} that in fact~$\mathcal{R}_q^{+} \simeq \mathcal{R}_q^{-}$ and this isomorphism is compatible with the obvious isomorphism~$\mathbb{Z}[q] \simeq \mathbb{Z}[q^{-1}]$. Hence we also have $\mathcal{R}_q = \mathcal{R}_q^{-} \otimes_{\mathbb{Z}[q^{-1}]} \mathbb{Z}[q]$ as subrings of~$\mathbb{Q}(q)[x]$ by Proposition~\ref{prop:localization}. Our main motivation for considering~$\mathcal{R}_q^{-}$ is that the ismorphism~$\mathcal{R}_q^{+} \simeq \mathcal{R}_q^{-}$ extends to an interesting involution of~$\mathcal{R}_q$.

Let us also briefly mention that in our setup there is another natural choice of generator~$z := (q-1)x + 1$ for $\mathbb{Q}(q)[x]$. Evaluating $x := [n]_q$ is the same as evaluating $z := q^n$ for all $n \in \mathbb{Z}$. And $x$ is obtainable from $z$ by the linear transformation $x = (q-1)^{-1}(z-1)$. Thus the ring of all polynomials $P \in \mathbb{Q}(q)[x]$ with $P(q^n) \in \mathbb{Z}[q]$ ($P(q^n) \in \mathbb{Z}[q,q^{-1}]$) for all $n \in \mathbb{N}$ ($n \in \mathbb{Z}$) is evidently isomorphic to $\mathcal{R}^{+}_q$ ($\mathcal{R}_q$). For $q$ specialized to a natural number this ``$z$-variable'' version of $\mathcal{R}_q^{+}$ is discussed in~\cite{gramain1990fonctions}. Going one step further, if we formally adjoin square roots~$K^2 = z$ and~$v^2 =q$, then the ring $\mathcal{R}_q[K,v]$ is equivalent to the Cartan part of Luztig's integral form of the quantum group $U_{v}(\mathfrak{sl}_2)$ \cite{lusztig1987modular}. Much of the theory developed here can be done just as easily in the $z$ or $K$ variable formulations of these rings. We chose the $x$ variable convention since it highlights the combinatorics and is the most transparent for specializing $q$ to $1$.

\end{section}
\begin{section}{A shift operator} \label{sec:shift}

The above proof of Lemma~\ref{lem:otherchars} relied on the fact that if $P(x) \in \mathcal{R}$ then the related polynomial~$P(x+1)$ is also in $\mathcal{R}$. In this section we study this operation of replacing $x$ by $x+1$ in more detail. So define the \emph{shift operator}~$S\colon \mathbb{Q}[x] \to \mathbb{Q}[x]$ to be the ring homomorphism given by~$S(x) := x+1$ and extended $\mathbb{Q}$-linearly. Note that $S$ is evidently invertible: we have~$S^{-1}(x) = x-1$. And also note that for $P \in \mathbb{Q}[x]$, we have $SP(n) = P(n+1)$ and $S^{-1}P(n) = P(n-1)$ for all~$n \in \mathbb{Z}$, which means that~$S$ restricts to an isomorphism $S\colon \mathcal{R}\to \mathcal{R}$. We can then ask how $S$ acts on the basis of binomial coefficient polynomials. The answer, thanks to Pascal's identity (the $q:=1$ specialization of~(\ref{eq:qpascal})) together with Lemma~\ref{lem:interpolation}, is that for all $k \in \mathbb{N}$
\begin{equation} \label{eq:binomshift}
S\binom{x}{k} = \binom{x}{k} + \binom{x}{k-1}
\end{equation}
with the convention that $\binom{x}{k} := 0$ for $k < 0$. From~(\ref{eq:binomshift}) it follows by induction that for all $m,k \in \mathbb{N}$ we have
\begin{equation} \label{eq:binommultishift}
S^m\binom{x}{k} = \sum_{i=0}^{m} \binom{m}{i}\binom{x}{k-i}.
\end{equation}
These shift operators appear in the translation between the binomial coefficient polynomial basis of $\mathcal{R}$ and the ``multichoose'' polynomial basis of $\mathcal{R}$. That is, if we set for all $k \in \mathbb{N}$
\[\multiset{x}{k} := S^{k-1}\binom{x}{k},\]
then it is well-known that for all $n\in\mathbb{N}$, $\multiset{n}{k}$ is the number of $k$-element multisets whose elements belong to $\{1,2,\ldots,n\}$ (see~\cite[pg.~26]{stanley1996ec1}). As we will see in Section~\ref{sec:involution}, we have $\binom{-x}{k} = (-1)^k\multiset{x}{k}$ and so the polynomials $(-1)^k\multiset{x}{k}$ for $k \in \mathbb{N}$ are also a $\mathbb{Z}$-basis for $\mathcal{R}$. A $q$-deformation of these multichoose polynomials will feature prominently in Section~\ref{sec:involution}, where we define a certain involution~\mbox{$\mathcal{R}_q \to \mathcal{R}_q$} (extending the involution~$\mathcal{R} \to \mathcal{R}$ given by $x \mapsto -x$) that restricts to an isomorphism~$\mathcal{R}_q^{+} \xrightarrow{\sim} \mathcal{R}_q^{-}$. But first we need a $q$-deformation of the shift operator. 

It is easy to $q$-deform the shift operator: we define $S\colon \mathbb{Q}(q)[x] \to\mathbb{Q}(q)[x]$ as the ring homomorphism with~$S(x) := qx + 1$ extended $\mathbb{Q}(q)$-linearly. Once again~$S$ is evidently invertible: $S^{-1}(x) = q^{-1}(x-1)$.  For $P \in \mathbb{Q}(q)[x]$, we have that~$SP([n]_q) = P([n+1]_q)$ and~$S^{-1}P([n]_q) = P([n-1]_q)$ for all~$n \in \mathbb{Z}$, which means that~$S$ restricts to an isomorphism $S\colon \mathcal{R}_q\to \mathcal{R}_q$. We can analogously ask how $S$ acts on the basis of $q$-binomial coefficient polynomials. The answer, thanks to the $q$-Pascal's identity~(\ref{eq:qpascal}) together with Lemma~\ref{lem:interpolation}, is that for all~$k \in \mathbb{N}$ we have
\begin{equation} \label{eq:qbinomshift}
S\qbinom{x}{k} = q^{k}\qbinom{x}{k} + \qbinom{x}{k-1},
\end{equation}
with the convention that $\qbinom{x}{k} := 0$ for $k < 0$. From~(\ref{eq:qbinomshift}) it follows by induction that for all $m,k \in \mathbb{N}$ we have
\begin{equation} \label{eq:qbinommultishift}
S^m\qbinom{x}{k} = \sum_{i=0}^{m} q^{(m-i)(k-i)}\qbinom{m}{i}_q \qbinom{x}{k-i}.
\end{equation}
Indeed, it is clear that~(\ref{eq:qbinommultishift}) holds for $m=0$. Then, supposing that~(\ref{eq:qbinommultishift}) holds for $m-1$, we have
\begin{align*}
S^m\qbinom{x}{k} &= \sum_{i=0}^{m-1} q^{(m-1-i)(k-i)}\qbinom{m-1}{i}_q S\qbinom{x}{k-i} \\
&= \sum_{i=0}^{m-1} q^{(m-1-i)(k-i)}\qbinom{m-1}{i}_q \left(q^{k-i}\qbinom{x}{k-i} + \qbinom{x}{k-(i+1)} \right) \\
&= \sum_{i=0}^{m} \left(q^{(m-1-i)(k-i)+(k-i)}\qbinom{m-1}{i}_q+q^{(m-i)(k-i+1)}\qbinom{m-1}{i-1}_q\right)\qbinom{x}{k-i} \\
&= \sum_{i=0}^{m} q^{(m-i)(k-i)}\left(\qbinom{m-1}{i}_q+q^{m-i}\qbinom{m-1}{i-1}_q\right)\qbinom{x}{k-i}\\
&= \sum_{i=0}^{m} q^{(m-i)(k-i)}\qbinom{m}{i}_q\qbinom{x}{k-i},
\end{align*}
as desired. In the last line of this computation we used the other $q$-Pascal's identity for $n,k\in\mathbb{N}$:
\[\qbinom{n}{k}_q = \qbinom{n-1}{k}_q + q^{n-k}\qbinom{n-1}{k-1}_q,\]
which follows from~(\ref{eq:qpascal}) by the symmetry~$\qbinom{n}{k}_q = \qbinom{n}{n-k}_q$. It is also straightforward to prove~(\ref{eq:qbinommultishift}) bijectively using Young diagrams. 

Lastly, we remark that $S$ restricts to a $\mathbb{Z}[q]$-linear map $S\colon \mathcal{R}_q^{+}\to \mathcal{R}_q^{+}$ and~$S^{-1}$ restricts to a $\mathbb{Z}[q^{-1}]$-linear map $S^{-1}\colon \mathcal{R}_q^{-}\to \mathcal{R}_q^{-}$, but these shift operators are no longer invertible when restricted to $\mathcal{R}_q^{+}$ or $\mathcal{R}_q^{-}$. However, we can easily describe the images of $\mathcal{R}_q^{+}$ and~$\mathcal{R}_q^{-}$ under the shift operator. For~$m \in \mathbb{Z}$ define
\begin{align*}
\mathcal{R}_q^{+,m} &:= \{P \in \mathbb{Q}(q)[x]\colon P([n]_q) \in \mathbb{Z}[q] \textrm{ for all } n\in\mathbb{Z}, n \geq m\};\\
\mathcal{R}_q^{-,m} &:= \{P \in \mathbb{Q}(q)[x]\colon P([n]_q) \in \mathbb{Z}[q^{-1}] \textrm{ for all } n\in\mathbb{Z}, n \leq m\}.
\end{align*}
Thus $R_q^{+} = R_q^{+,0}$, $R_q^{-} = R_q^{-,0}$, and
\begin{align*}
\cdots \subseteq \mathcal{R}_q^{+,-2} \subseteq \mathcal{R}_q^{+,-1} &\subseteq \mathcal{R}_q^{+} \subseteq \mathcal{R}_q^{+,1} \subseteq \mathcal{R}_q^{+,2} \subseteq \cdots; \\
\cdots \supseteq \mathcal{R}_q^{-,-2} \supseteq \mathcal{R}_q^{-,-1} &\supseteq \mathcal{R}_q^{-} \supseteq \mathcal{R}_q^{-,1} \supseteq \mathcal{R}_q^{-,2} \supseteq \cdots.
\end{align*}

\begin{proposition}
For $m \in \mathbb{Z}$, $S^m\colon \mathcal{R}_q^{+} \xrightarrow{\sim} \mathcal{R}_q^{+,-m}$ is an isomorphism of~$\mathbb{Z}[q]$-algebras, and $S^{-m}\colon \mathcal{R}_q^{-} \xrightarrow{\sim} \mathcal{R}_q^{-,m}$ is an isomorphism of~$\mathbb{Z}[q^{-1}]$-algebras.
\end{proposition}
\begin{proof}
We check that $S^{m}\mathcal{R}_q^{+} \subseteq  \mathcal{R}_q^{+,-m}$ and $S^{-m}\mathcal{R}_q^{+,-m} \subseteq \mathcal{R}_q^{+}$: for~$P(x) \in \mathcal{R}_q^{+}$, we have $S^{m}P([n]_q) = P([n+m]_q) \in \mathbb{Z}[q,q^{-1}]$ for all $n \in \mathbb{Z}$ with $n \geq m$, so indeed~$S^{m}P \in \mathcal{R}_q^{+,-m}$; similarly, for~$P(x) \in \mathcal{R}_q^{+,-m}$, we have $S^{-m}P([n]_q) = P([n-m]_q) \in \mathbb{Z}[q,q^{-1}]$ for all $n \in \mathbb{N}$, so indeed $S^{-m}P \in \mathcal{R}_q^{+}$. The statement about $\mathcal{R}_q^{-}$ is analogous.
\end{proof}

\end{section}

\begin{section}{A bar involution} \label{sec:involution}
The ring $\mathbb{Z}[q,q^{-1}]$ of Laurent polynomials has an obvious $\mathbb{Z}$-linear involution given by~$q \mapsto q^{-1}$. This involution is fundamental in Kazhdan-Lusztig theory~\cite{kazhdan1979representations}, where it is extended to a \emph{bar involution} (or \emph{bar operator}) of the Hecke algebra of a Coxeter group. Thus we refer to the involution $q \mapsto q^{-1}$ of~$\mathbb{Z}[q,q^{-1}]$ as the \emph{bar involution} $\overline{\raisebox{0.5em}{\;\;}}\colon \mathbb{Z}[q,q^{-1}]\to \mathbb{Z}[q,q^{-1}]$ as well. Accordingly, we write $\overline{c}$ for the result of applying this involution to an element $c \in \mathbb{Z}[q,q^{-1}]$. We will now describe an extension of the bar involution to~$\mathcal{R}_q$. 

First note that $\overline{\raisebox{0.5em}{\;\;}}\colon \mathbb{Z}[q,q^{-1}]\to \mathbb{Z}[q,q^{-1}]$ extends uniquely to a field involution~$\overline{\raisebox{0.5em}{\;\;}}\colon \mathbb{Q}(q) \to \mathbb{Q}(q)$. And there is a unique extension of~$\overline{\raisebox{0.5em}{\;\;}}\colon \mathbb{Q}(q) \to \mathbb{Q}(q)$ to a ring involution~$\overline{\raisebox{0.5em}{\;\;}}\colon \mathbb{Q}(q)[x] \to \mathbb{Q}(q)[x]$ satisfying $\overline{x} := -qx$. If we write any~$P(x) \in \mathbb{Q}(q)[x]$ as $P(x) = c_0 + c_1x+\cdots + c_kx^k$ with $c_i \in \mathbb{Q}(q)$, then for all~$n \in \mathbb{Z}$ we have~$\overline{P}([-n]_q) = \overline{c_0} + \overline{c_1}\overline{[n]_{q}} + \cdots + \overline{c_k}(\overline{[n]_{q}})^{k}$, which means that~$\overline{\overline{P}([-n]_q)} = P([n]_q)$ for all $n \in \mathbb{Z}$. In other words, we have the following commutative diagram for all $n \in \mathbb{Z}$: \\
\centerline{
\xymatrix{
\mathbb{Q}(q)[x] \ar[d]_{P(x) \mapsto \overline{P(x)}} \ar[rr]^{x\mapsto [n]_q}& & \mathbb{Q}(q) \ar[d]^{c\mapsto \overline{c}} \\
\mathbb{Q}(q)[x] \ar[rr]^{x\mapsto [-n]_q}& & \mathbb{Q}(q)
}}
Now suppose that $P(x) \in \mathcal{R}_q$. Thus $P([n]_q) \in \mathbb{Z}[q,q^{-1}]$ for all $n \in \mathbb{Z}$. So as a result of the above diagram, $\overline{P}([-n]_q) \in \overline{\mathbb{Z}[q,q^{-1}]} = \mathbb{Z}[q,q^{-1}]$ for all $n \in \mathbb{Z}$. Thus~$\overline{\raisebox{0.5em}{\;\;}}\colon \mathbb{Q}(q)[x] \to \mathbb{Q}(q)[x]$ restricts to a ring involution~$\overline{\raisebox{0.5em}{\;\;}}\colon \mathcal{R}_q \to \mathcal{R}_q$, which again we call the bar involution. Also as a result of the above commutative diagram, the bar involution restricts to an isomorphism $\overline{\raisebox{0.5em}{\;\;}}\colon \mathcal{R}_q^{+}\xrightarrow{\sim} \mathcal{R}_q^{-}$ that respects the obvious isomorphism $\mathbb{Z}[q] \xrightarrow{\sim} \mathbb{Z}[q^{-1}]$.

The relationship between the shift operator and the bar involution is as follows.

\begin{proposition} \label{prop:otherinvs}
For all $P(x) \in \mathcal{R}_q$, we have~$\overline{SP(x)} = S^{-1}\overline{P(x)}$. Thus for all $m \in \mathbb{Z}$, the map~$P(x) \mapsto S^{m}\overline{P(x)}$ is an involution $\mathcal{R}_q \to \mathcal{R}_q$.
\end{proposition}
\begin{proof}
For the first statement, by Lemma~\ref{lem:interpolation} we need only check that 
\[\overline{SP(x)}\rvert_{x:=[n]_q} = S^{-1}\overline{P(x)}\rvert_{x:=[n]_q}\] 
for all $n \in \mathbb{N}$. But by what we already know about the shift operator and the bar involution, these are both equal to $P([-n-1]_q)$. As for the second statement: by the first statement, we have $S^m\overline{S^m\overline{P(x)}} = \overline{S^{-m}S^{m}\overline{P(x)}} = \overline{\overline{P(x)}} = P(x)$ for all~$m \in \mathbb{Z}$ and all~$P(x) \in \mathcal{R}_q$.
\end{proof}

We also have the following corollary of the existence of the bar involution, giving a distinguished basis for $\mathcal{R}_q^{-}$.

\begin{corollary} \label{cor:negqbinom}
$\mathcal{R}_q^{-}$ is freely generated as a $\mathbb{Z}[q^{-1}]$-module by~$\overline{ \qbinom{x}{k}}$ for $k\in \mathbb{N}$. For $i,j \in \mathbb{N}$ these basis elements multiply as
\[ \overline{ \qbinom{x}{k}} \; \overline{ \qbinom{x}{j}} = \sum_{k = \mathrm{max}(i,j)}^{i+j}  \frac{q^{i(i-k)+j(j-k)} \, [k]_q!}{[k-i]_q![k-j]_q![i+j-k]_q!} \overline{ \qbinom{x}{k}}.\]
In particular, the $\mathbb{Z}[q^{-1}]$-algebra $\mathcal{R}^{-}_q$ with distinguished basis $\{\overline{ \qbinom{x}{k}}\colon k \in \mathbb{N}\}$ has structure constants in~$\mathbb{N}[q^{-1}]$.
\end{corollary}
\begin{proof}
That the $\overline{ \qbinom{x}{k}}$ give a $\mathbb{Z}[q^{-1}]$-basis for $\mathcal{R}_q^{-}$ follows immediately from the bar involution isomorphism. The formula for the structure constants is obtained by applying the bar involution to the formula in Theorem~\ref{thm:qbinommult}. Note especially that we have~$\overline{[n]_q} = q^{-(n-1)}[n]_q$ for all $n\in\mathbb{Z}$.
\end{proof}

As a consequence of Proposition~\ref{prop:localization} together with Proposition~\ref{prop:qbinom} and Corollary~\ref{cor:negqbinom}, both $\qbinom{x}{k}$ and $\overline{ \qbinom{x}{k}}$ are $\mathbb{Z}[q,q^{-1}]$-bases of~$\mathcal{R}_q$. So it makes sense to ask how to write one of these bases is terms of the other. The answer is given by the following proposition.

\begin{proposition} \label{prop:invconsts}
For all $k \in \mathbb{N}$ we have
\begin{align*}
\overline{ \qbinom{x}{k}} &= (-1)^k q^{\binom{k+1}{2}}S^{k-1}\qbinom{x}{k} \\
&= (-1)^k \sum_{i=0}^{k-1} q^{\binom{k+1}{2}+(k-1-i)(k-i)} \qbinom{k-1}{i}_q \qbinom{x}{k-i}
\end{align*}
In particular the coefficients expressing $\overline{ \qbinom{x}{k}}$ in the basis of $\qbinom{x}{i}$ are in~$(-1)^k\mathbb{N}[q]$, which means that in fact $\overline{ \qbinom{x}{k}} \in \mathcal{R}_q^{+}$.

For all $k \in \mathbb{N}$ we also have
\begin{align*}
\qbinom{x}{k} &= (-1)^k q^{-\binom{k+1}{2}}S^{-(k-1)}\overline{ \qbinom{x}{k}} \\
&=(-1)^k \sum_{i=0}^{k-1} q^{-\binom{k+1}{2}+(i-k+1)k} \qbinom{k-1}{i}_q \overline{ \qbinom{x}{k-i}},
\end{align*}
In particular the coefficients expressing~$\qbinom{x}{k}$ in the basis of $\overline{ \qbinom{x}{i}}$ are in~$(-1)^k\mathbb{N}[q^{-1}]$, which means that in fact $\qbinom{x}{k} \in \mathcal{R}_q^{-}$.

\end{proposition}
\begin{proof}
Let us first address how to express $\overline{ \qbinom{x}{k}}$ in the basis of $\qbinom{x}{i}$. The second equality follows from the first by an application of~(\ref{eq:qbinommultishift}). For the first equality, by Lemma~\ref{lem:interpolation} it suffices to verify
\[ \left.\overline{ \qbinom{x}{k}}\right\rvert_{x:=[n]_q} = (-1)^k q^{\binom{k+1}{2}}\left.S^{k-1}\qbinom{x}{k}\right\rvert_{x:=[n]_q}  \]
for all $n \in \mathbb{N}$. This follows from straightforward algebraic manipulation: we verify directly from the definitions of $\qbinom{x}{k}$, the bar involution, and the shift operator (as well as the fact that~$\overline{ [n]_q} = q^{-(n-1)}[n]_q$ for all $n\in\mathbb{Z}$) that
\begin{align*}
\left.\overline{ \qbinom{x}{k}}\right\rvert_{x:=[n]_q} &= \left.\overline{\prod_{i=0}^{k-1} \frac{x-[i]_q}{q^i [i+1]_q}}\right\rvert_{x:=[n]_q} \\
&=  \left.\prod_{i=0}^{k-1} \frac{-qx-q^{-(i-1)}[i]_q}{q^{-2i} [i+1]_q}\right\rvert_{x:=[n]_q} \\
&=\prod_{i=0}^{k-1} \frac{-q[n]_q-q^{-(i-1)}[i]_q}{q^{-2i} [i+1]_q} \\
&=(-1)^k \prod_{i=0}^{k-1} q^{i+1} \frac{[n+i]_q}{[i+1]_q} \\
&=(-1)^k q^{\binom{k+1}{2}} \prod_{i=0}^{k-1}  \frac{[n+k-1]_q-[i]_q}{q^{i}[i+1]_q} \\
&= (-1)^k q^{\binom{k+1}{2}}\left.S^{k-1}\qbinom{x}{k}\right\rvert_{x:=[n]_q},
\end{align*}
as desired.

Now let us address how to express $\qbinom{x}{k}$ in the basis of $\overline{ \qbinom{x}{i}}$. The first equality follows from the claim about how to express $\overline{ \qbinom{x}{k}}$ in the basis of $\qbinom{x}{i}$ by applying~$S^{-(k-1)}$ to both sides. Then by applying the bar involution to both sides of~(\ref{eq:qbinommultishift}), and using the fact proved in Proposition~\ref{prop:otherinvs} that $\overline{SP} = S^{-1}\overline{P}$, we get for $m,k \in \mathbb{N}$ that
\begin{equation} \label{eq:negqbinommultishift}
S^{-m}\overline{ \qbinom{x}{k}} = \sum_{i=0}^{m} q^{(i-m)k} \qbinom{m}{i}_q \qbinom{x}{k-i}.
\end{equation}
Equation~(\ref{eq:negqbinommultishift}) lets us deduce the second equality from the first.
\end{proof}

It is worth remarking, as mentioned in Section~\ref{sec:shift}, that the~$q:=1$ case of Proposition~\ref{prop:invconsts} says that
\[ \binom{-n}{k}=(-1)^k\multiset{n}{k} \]
for all $n,k \in \mathbb{N}$. This duality between ``$n$ choose $k$'' and ``$n$ multichoose $k$,'' an observation which has been attributed to Riordan~\cite{riordan1958introduction}, is the starting point for the study of combinatorial reciprocity theorems~\cite{stanley1974reciprocity}. We also note that it is possible to give a combinatorial interpretation for~$\overline{\qbinom{n}{k}}$ with~$n,k \in \mathbb{N}$ as a generating function for certain lattice paths by area under the path, generalizing the multiset interpretation of $(-1)^k\binom{-n}{k}$.

\end{section}

\begin{section}{Lucas' theorem and a quantum Frobenius map} \label{sec:frobenius}

We'd now like to define a quantum Frobenius map on (a base change of)~$\mathcal{R}_q$.  To highlight the analogy we will first review the usual Frobenius map on~$\mathcal{R}\otimes_{\mathbb{Z}} \mathbb{F}_p$.  We recall the celebrated Lucas' theorem on binomial coefficients.

\begin{theorem}[Lucas 1878~\cite{lucas1878theorie}]\label{thm:lucas}
Let $p$ be a prime. Let $n,m \in \mathbb{N}$. Suppose that~$n = n_0 + n_1p+ \dots + n_kp^k$ and~$m = m_0 + m_1p+\dots + m_kp^k$ are the base $p$ expansions for $n$ and $m$ (so $n_i,m_i \in [0,p-1]$ for all $1 \leq i \leq k$). Then,
\[\binom{n}{m} \equiv \binom{n_0}{m_0} \binom{n_1}{m_1} \binom{n_2} {m_2} \dots \binom{n_k}{m_k}  \mod p.\]
\end{theorem}

We have the following corollary of Lucas' theorem, proving the existence of a Frobenius map for $\mathcal{R} \otimes_{\mathbb{Z}} \mathbb{F}_p$.

\begin{corollary} Let $p$ be a prime.
\begin{enumerate}
\item The map $\Psi_p: \mathcal{R} \otimes_{\mathbb{Z}} \mathbb{F}_p\to \mathcal{R} \otimes_{\mathbb{Z}} \mathbb{F}_p$ defined by $\Psi_p\colon\binom{x}{k} \mapsto \binom{x}{pk}$ and extended $\mathbb{F}_p$-linearly is a ring homomorphism.

\item $\Psi_p$ admits a one-sided inverse $\widetilde{\Psi}_p: \mathcal{R} \otimes_{\mathbb{Z}} \mathbb{F}_p\to \mathcal{R} \otimes_{\mathbb{Z}} \mathbb{F}_p$ defined by
\[ \widetilde{\Psi}_p\left(\binom{x}{k}\right) := \begin{cases} 
      \binom{x}{k/p} & \textrm{ if $p\mid k$}; \\
      0 & \textrm{ otherwise}. \\
   \end{cases} \]
\end{enumerate}
\end{corollary}

\begin{proof} 
All we need to do is check these formulas are compatible with the multiplication formula in Theorem~\ref{thm:binommult}.  So let us expand~$\Psi_p(\binom{x}{i})\Psi_p(\binom{x}{j})$ in $\mathcal{R}\otimes_{\mathbb{Z}}\mathbb{F}_p$:
\[\Psi_p\left(\binom{x}{i}\right)\Psi_p\left(\binom{x}{j}\right) = \binom{x}{pi}\binom{x}{pj} = \sum_{k = \mathrm{max}(pi,pj)}^{pi+pj}\binom{k}{pi}\binom{pi}{pj+pi-k}\binom{x}{k}.\]
Now Lucas' theorem tells us that $\binom{pi}{pj+pi-k} \equiv 0 \mod p$ unless $p\mid k$.  So throwing out those terms that vanish we may rewrite this with $k = pk'$ as
\[\Psi_p\left(\binom{x}{i}\right)\Psi_p\left(\binom{x}{j}\right) = \sum_{k' = \mathrm{max}(i,j)}^{i+j}\binom{pk'}{pi}\binom{pi}{pj+pi-pk'}\binom{x}{pk'}.\]
Using the definition of $\Psi_p$ and simplifying with Lucas' theorem one more time we get
\[\Psi_p\left(\binom{x}{i}\right)\Psi_p\left(\binom{x}{j}\right) = \sum_{k' = \mathrm{max}(i,j)}^{i+j}\binom{k'}{i}\binom{i}{j+i-k'}\Psi_p\left(\binom{x}{k'}\right),\]
which is exactly $\Psi_p\left(\binom{x}{j}\binom{x}{k}\right)$ according to Theorem~\ref{thm:binommult}. Hence $\Psi_p$ is a ring homomorphism, as desired.

For the second part, above calculation shows that $\widetilde{\Psi}_p$ is multiplicative for those basis vectors it does not send to zero. So all that remains is to check that the span of the polynomials $\binom{x}{i}$ with $p \nmid i$ is an ideal of $\mathcal{R}\otimes \mathbb{F}_p$. This again can be seen directly from Theorem~\ref{thm:binommult} and Lucas' theorem. Let $i \in \mathbb{N}$ satisfy~$p \nmid i$ and $j \in \mathbb{N}$ be arbitrary. We have that
\[\binom{x}{i}\binom{x}{j} = \sum_{k=\mathrm{max}(i,j)}^{i+j}\binom{k}{i}\binom{i}{i+j-k}\binom{x}{k}.\]
Lucas' theorem tells us that if $p \mid k$ then $\binom{k}{i} \equiv 0 \mod p$, hence we may take the sum to just be over those $k$ such that $p \nmid k$, as desired.
\end{proof}

The Frobenius map $x \mapsto x^p$ defined on $\mathbb{F}_p[x]$ commutes with the shift operator $x \mapsto x+1$ and with the ``bar involution'' $x \mapsto -x$. However, note crucially that $\mathcal{R} \otimes_{\mathbb{Z}} \mathbb{F}_p$ does not naturally sit inside $\mathbb{F}_p[x]$ and that moreover~$\Psi_p\colon \mathcal{R} \otimes_{\mathbb{Z}} \mathbb{F}_p \to \mathcal{R} \otimes_{\mathbb{Z}} \mathbb{F}_p$ does not extend $x \mapsto x^p$. Indeed, the Frobenius map~$\Psi_p\colon \mathcal{R} \otimes_{\mathbb{Z}} \mathbb{F}_p \to \mathcal{R} \otimes_{\mathbb{Z}} \mathbb{F}_p$ does not appear to have a simple relationship to the shift operator $S\colon \mathcal{R} \otimes_{\mathbb{Z}} \mathbb{F}_p \to \mathcal{R} \otimes_{\mathbb{Z}} \mathbb{F}_p$ or the bar involution~$\overline{\raisebox{0.5em}{\;\;}}\colon \mathcal{R} \otimes_{\mathbb{Z}} \mathbb{F}_p \to \mathcal{R} \otimes_{\mathbb{Z}} \mathbb{F}_p$. (These maps are obtained from the ones defined on~$\mathcal{R}_q$ by specializing $q:=1$ and then tensoring with~$\mathbb{F}_p$.) However, we do have the following proposition which says that $\Psi_p$ commutes with one of the related involutions from Proposition~\ref{prop:otherinvs}.

\begin{proposition} \label{prop:frobcommute}
For all $P(x) \in \mathcal{R} \otimes_{\mathbb{Z}} \mathbb{F}_p$, we have $\Psi_p(S\overline{P(x)}) = S\overline{\Psi_p(P(x))}$.
\end{proposition}
\begin{proof}
It suffices to verify that $\Psi_p(S\overline{ \binom{x}{k}}) = S\overline{ \Psi_p(\binom{x}{k})}$ for all $k \in \mathbb{N}$ because the~$\binom{x}{k}$ are a $\mathbb{Z}$-basis of $\mathcal{R}$ and both expressions are clearly $\mathbb{Z}$-linear. To that end, using Proposition~\ref{prop:invconsts} and Equation~(\ref{eq:binommultishift}) we compute that
\begin{align*}
S\overline{ \Psi_p\left(\binom{x}{k}\right)} &= S\overline{ \binom{x}{pk}} \\
&= (-1)^{pk} S^k\binom{x}{pk} \\
&= (-1)^{pk} \sum_{i=0}^{pk} \binom{pk}{i} \binom{x}{pk-i}
\end{align*}
We can simplify this expressing, noting first of all that $(-1)^{pk} \equiv (-1)^{k} \mod p$, and also, thanks to Lucas' theorem, that $\binom{pk}{i} \equiv 0 \mod p$ if $p \nmid i$. Ignoring the terms that vanish and writing $k =pk'$ we have
\[S\overline{ \Psi_p\left(\binom{x}{k}\right)} = (-1)^k \sum_{i'=0}^{k} \binom{pk}{pi'} \binom{x}{pk-pi'}.\]
Again applying Lucas' theorem, we get
\[S\overline{ \Psi_p\left(\binom{x}{k}\right)} = (-1)^k \sum_{i'=0}^{k} \binom{k}{i'} \binom{x}{p(k-i')},\]
which is exactly $\Psi_p(S\overline{ \binom{x}{k}})$ according to Proposition~\ref{prop:invconsts} and Equation~(\ref{eq:binommultishift}).
\end{proof}

Now let us try to extend the above to $\mathcal{R}_q$. A naive thing to try would be to lift this to a map from $\mathcal{R}_q \otimes \mathbb{F}_p$ with $\qbinom{x}{k} \mapsto \qbinom{x}{pk}$.  However looking at the degree in $q$ of the multiplicative constants it is clear that such a map cannot be a ring homomorphism.

Instead, the connection to representation theory suggest that there should be certain similarities between working in positive characteristic at $q:=1$, and specializing $q$ to a root of unity. Indeed, we can generalize the above argument to define a quantum Frobenius map on certain quotients of $\mathcal{R}_q$, but first we will need the following $q$-analog of Lucas' theorem due (we believe) to Sved henceforth referred to as the $q$-Lucas' theorem.

\begin{theorem}[Sved 1988~\cite{sved1988divisibility}]\label{thm:qlucas}
Let $d$ be a positive integer and $n,m \in \mathbb{N}$. Suppose that~$n = dn'+n_0$ and~$m = dm'+m_0$ with~$n_0,m_0 \in [0,d-1]$.  Then
\[\qbinom{n}{m}_q \equiv \binom{n'}{m'}\qbinom{n_0}{m_0}_q \mod \Phi_d(q),\]
where $\Phi_d$ denotes the $d$th cyclotomic polynomial.
\end{theorem}

Comparing this to the usual Lucas' theorem, this suggests that we should look for quantum Frobenius maps not from $\mathcal{R}_q / \Phi_d(q)$ to itself, but between~$\mathcal{R}$ and~$\mathcal{R}_q / \Phi_d(q)$. Specifically, we have the following corollary to the $q$-Lucas' theorem.

\begin{corollary} Let $d$ be a positive integer.
\begin{enumerate}
\item The map $\Psi_d: \mathcal{R} \to \mathcal{R}_q / \Phi_d(q)$ defined by $\binom{x}{k} \mapsto \qbinom{x}{dk}$ and extended $\mathbb{Z}$-linearly is a ring homomorphism.

\item After extending scalars, $\Psi_d$ admits a one sided inverse $\widetilde{\Psi}_d: \mathcal{R}_q / \Phi_d(q) \to \mathcal{R} \otimes_{\mathbb{Z}} \mathbb{Z}[q,q^{-1}]/{\Phi_d(q)}$ defined by

\[ \widetilde{\Psi}_d\left(\qbinom{x}{k}\right) := \begin{cases} 
      \binom{x}{k/d} & \textrm{ if $d\mid k$}; \\
      0 & \textrm{ otherwise}. \\
   \end{cases} \]

\end{enumerate}
\end{corollary}

\begin{proof}
As before for the first part it suffices to check compatibility with the multiplication formulas from Theorems~\ref{thm:binommult} and~\ref{thm:qbinommult}. So we compute
\[\Psi_d\left(\binom{x}{i}\right)\Psi_d\left(\binom{x}{j}\right) =\qbinom{x}{di}\qbinom{x}{dj} = \hspace{-0.5cm} \sum_{k = \mathrm{max}(di,dj)}^{di+dj} \hspace{-0.5cm} q^{(k-di)(k-dj)} \qbinom{k}{di}_q \qbinom{di}{di+dj-k}_q \qbinom{x}{k} .\]
By the $q$-Lucas' theorem, the term $\qbinom{di}{di+dj-k}_q$ vanishes modulo~$\Phi_d(q)$ unless $d$ divides~$k$. Ignoring the terms that vanish and writing $k = dk'$ we have
\[\Psi_d\left(\binom{x}{i}\right)\Psi_d\left(\binom{x}{j}\right) = \hspace{-0.5cm}  \sum_{k' = \mathrm{max}(i,j)}^{i+j} \hspace{-0.5cm} q^{(dk'-di)(dk'-dj)} \qbinom{dk'}{di}_q \qbinom{di}{di+dj-dk'}_q \qbinom{x}{dk'}  .\]
We can simplify further the above expression using the $q$-Lucas' theorem and the fact that $q^d=1$ modulo $\Phi_d(q)$ to get
\[\Psi_d\left(\binom{x}{i}\right)\Psi_d\left(\binom{x}{j}\right) = \sum_{k' = \mathrm{max}(i,j)}^{i+j}  \binom{k'}{i} \binom{i}{i+j-k'} \qbinom{x}{dk'},\]
which is exactly $\Psi_d(\binom{x}{i}\binom{x}{j})$ according Theorem~\ref{thm:binommult}. Hence $\Psi_d$ is a ring homomorphism, as desired.

For the second part, as before above calculation shows $\widetilde{\Psi}_d$ is multiplicative for those basis vectors it does not send to zero. Thus we just need to check that the span of the $\qbinom{x}{i}$ with $d\nmid i$ forms an ideal. As before take $i \in \mathbb{N}$ such that~$d\nmid i$ and let $j \in \mathbb{N}$ be arbitrary. By Theorem~\ref{thm:qbinommult} we have
\[\qbinom{x}{i}\qbinom{x}{j} = \sum_{k = \mathrm{max}(i,j)}^{i+j}  q^{(k-i)(k-j)} \qbinom{k}{i}_q \qbinom{i}{i+j-k}_q  \qbinom{x}{k}.\]
The $q$-Lucas' theorem tells us that if $d \mid k$ then $\qbinom{k}{i}_q = 0$ modulo~$\Phi_d(q)$ and hence we may rewrite this as a sum over those $k$ with $d \nmid k$, as desired.
\end{proof}

The direct analog of Proposition~\ref{prop:frobcommute} holds in this $q$ a root of unity case.

\begin{proposition} \label{prop:qfrobcommute}
For all $P(x) \in \mathcal{R}$, we have $\Psi_d(S\overline{P(x)}) = S\overline{\Psi_d(P(x))}$.
\end{proposition}
\begin{proof}
It suffices to prove that $\Psi_d(S\overline{ \binom{x}{k}}) = S\overline{ \Psi_d(\binom{x}{k})}$ for all $k \in \mathbb{N}$ since both expressions are $\mathbb{Z}$-linear. To that end, using Proposition~\ref{prop:invconsts} and Equation~(\ref{eq:qbinommultishift}), we compute
\begin{align*}
S\overline{ \Psi_d\left(\binom{x}{k}\right)} &= S\overline{ \qbinom{x}{dk}} \\
&= (-1)^{dk} q^{\binom{dk+1}{2}} S^{dk}\qbinom{x}{dk} \\
&= (-1)^{dk} q^{\binom{dk+1}{2}} \sum_{i=0}^{dk} q^{(dk-i)(dk-i)}\qbinom{dk}{i}_q\qbinom{x}{dk-i}.
\end{align*}
We claim that $(-1)^{dk} q^{\binom{dk+1}{2}} = (-1)^k$ modulo~$\Phi_d(q)$. If $d$ is even, then we have~$q^{d/2} = (-1)$ modulo $\Phi_d(q)$ so $(-1)^{dk} q^{\binom{dk+1}{2}} = (-1)^{dk^2+dk+k} = (-1)^k$ modulo~$\Phi_d(q)$, where we use the fact that $dk(k+1)$ is even since $k(k+1)$ is even. If $d$ is odd then $d$ divides $\binom{dk+1}{2}$ so $q^{\binom{dk+1}{2}} = 1$ modulo~$\Phi_d(q)$ and we have~$(-1)^{dk}q^{\binom{dk+1}{2}} = (-1)^{dk} = (-1)^{k}$ modulo~$\Phi_d(q)$. Also note that by the $q$-Lucas' theorem $\qbinom{dk}{i}_q = 0$ modulo $\Phi_d(q)$ unless $d \mid i$. Ignoring the terms that vanish and writing $i = di'$, the above becomes
\[S\overline{ \Psi_d\left(\binom{x}{k}\right)} = (-1)^k \sum_{i'=0}^{k}q^{(dk-di')(dk-di')}\qbinom{dk}{di'}_q\qbinom{x}{d(k-i')}.\]
Now we can use the fact that $q^d=1$ modulo $\Phi_d(q)$ and apply the $q$-Lucas' theorem again to get
\[S\overline{ \Psi_d\left(\binom{x}{k}\right)} = (-1)^k \sum_{i'=0}^{k}\binom{k}{i'}\qbinom{x}{d(k-i')},\]
which is exactly $\Psi_d(S\overline{ \binom{x}{k}})$ according to Proposition~\ref{prop:invconsts} and Equation~(\ref{eq:binommultishift}).
\end{proof}

\end{section}

\begin{section}{Classification of maps into a field} \label{sec:mapstofields}

A basic problem one can pose for any commutative ring is to classify homomorphisms from that ring into fields. This problem is closely related to the problem of classifying the points of the spectrum of the ring, i.e., the prime ideals of the ring. Indeed, the prime ideals of a commuatitive ring are precisely the kernels of maps to fields (although the correspondence is not one-to-one, due to the existence of injective maps between fields).

The maximal ideals (i.e.,~kernels of surjective maps to fields) of $\mathcal{R}$ were classified by Brizolis~\cite{brizolis1976ideals} in 1976: for $p$ a prime and $t \in \mathbb{Z}_p$ a $p$-adic integer, the following set of polynomials (where $\lvert\cdot\rvert_p$ denotes the $p$-adic norm)
\[ M_{p,t} := \{P(x) \in \mathcal{R}\colon \lvert P(t) \rvert_p < 1\}\]
is a maximal ideal of $\mathcal{R}$; moreover, any maximal ideal $M$ of $\mathcal{R}$ is $M=M_{p,t}$ for some $p$ and $t$, and $M_{p,t}=M_{p',t'}$ if and only if $p=p'$ and $t=t'$. An intriguing consequence of this classification is that while $\mathbb{Z}[x]\subsetneq \mathcal{R} \subsetneq \mathbb{Q}[x]$ and~$\mathbb{Z}[x]$ and~$\mathbb{Q}[x]$ both have countably many maximal ideals, $\mathcal{R}$ has uncountably many maximal ideals.

More recently, the first author used a classification the maps~$\mathcal{R}\to \mathbf{k}$ for~$\mathbf{k}$ a field of positive characteristic as part of his investigation of stability properties of the modular representation theory of symmetric groups~\cite{harman2015stability}. We now  extend this classification to the quantum setting. From now on in this section, fix an arbitrary field $\mathbf{k}$. In Theorem~\ref{thm:mapstofields} below we will classify all ring homomorphisms~$\mathcal{R}_q^{+}\to \mathbf{k}$, breaking up the classification into cases of where $q$ is sent. In the process we also classify all homomorphisms~$\mathcal{R}_q\to \mathbf{k}$, which are the same except that we forbid $q:=0$.

As a first source of homomorphisms $\mathcal{R}^{+}_q \to \mathbf{k}$ we have the following ``standard'' evaluation maps: first we specialize $x:=[n]_q$ for some $n \in \mathbb{N}$ to get a homomorphism from $\mathcal{R}^{+}_q$ to~$\mathbb{Z}[q]$ and then we compose with a map from~$\mathbb{Z}[q]$ to~$\mathbf{k}$ defined by sending~$q$ to any~$\kappa \in \mathbf{k}$. Let us call this map $\mathrm{std}_{n,\kappa}\colon \mathcal{R}_q^{+} \to \mathbf{k}$. These standard maps are certainly not all the maps from $\mathcal{R}_q^{+}$ into $\mathbf{k}$, but as we shall see they are ``dense'' in the set of such maps.

To describe all the maps we need some notation. If $\mathrm{char}(\mathbf{k}) = 0$, then we have~$\mathcal{R} \otimes_{\mathbb{Z}} \mathbf{k} \subseteq \mathbf{k}[x]$ in a natural way. Thus we may treat any $P(x) \in \mathcal{R}$ as element of $\mathbf{k}[x]$. In particular, if $t \in \mathbf{k}$ then for every~$m \in \mathbb{N}$ we define~$\binom{t}{m}$ to be the result of evaluating $\binom{x}{m} \in \mathbf{k}[x]$ at $x := t$. On the other hand, suppose for a moment that $\mathrm{char}(\mathbf{k}) = p > 0$. Then $\mathcal{R} \otimes_{\mathbb{Z}} \mathbf{k}$ does not naturally sit inside of~$\mathbf{k}[x]$, so it does not make sense to evaluate $\binom{x}{m}$ at an arbitrary element~$t \in \mathbf{k}$. But if~$t \in \mathbb{Z}_p$ is a $p$-adic integer, then $t$ has a base $p$ expansion $t = t_0 + t_1p + t_2p^2 +\cdots$. Thus, following Lucas' theorem  (Theorem~\ref{thm:lucas}), we can in this case define
\[\binom{t}{m} := \binom{t_0}{m_0}\binom{t_1}{m_1}\cdots\binom{t_k}{m_k} \mod p\]
for every $m \in \mathbb{N}$ with base $p$ expansion $m=m_0 + m_1p + \cdots + m_kp^k$. Note that according to this definition, $\binom{t}{m} \in \mathbb{F}_p$. But $\mathbb{F}_p \subseteq \mathbf{k}$ in a unique way, so we can in fact treat $\binom{t}{m}$ as an element of $\mathbf{k}$.

Now return to considering general $\mathbf{k}$. The last paragraph discussed evaluating binomial coefficient polynomials to obtain elements of $\mathbf{k}$. What about evaluating $q$-binomial coefficient polynomials to obtain elements of $\mathbf{k}$? As long as the denominator~$q^{\binom{m}{2}}[m]_q!$ of $\qbinom{x}{m}$ is not zero when we set~$q:= \kappa \in \mathbf{k}$, the polynomial~$\qbinom{x}{m}_{q:=\kappa}$ is a well-defined element of $\mathbf{k}[x]$. In this case we can of course then define~$\qbinom{t}{m}_{q:=\kappa} \in \mathbf{k}$ to be the result of evaluating $\qbinom{x}{m}_{q:=\kappa} \in \mathbf{k}[x]$ at some element~$x := t \in \mathbf{k}$. So, for a given $m\in\mathbb{N}$, when does $q^{\binom{m}{2}}[m]_q!$ evaluated at~$q:= \kappa \in \mathbf{k}$ equal zero? Exactly under the folloxwing circumstances:
\begin{itemize}
\item $\kappa = 0$ and $m \geq 2$;
\item $\kappa = 1$ and $m \geq p$ where $\mathrm{char}(\mathbf{k}) = p > 0$;
\item $\kappa$ is a primitive $d$th root of unity for some $d > 1$ and $m \geq d$.
\end{itemize}
This claim is easily verified: the $q=0$ and $q=1$ cases are clear; and if $q \neq 1,0$ then~$q^{\binom{m}{2}}[m]_q!$ equals zero if and only if
\[(q-1)^{m-1}q^{-\binom{m}{2}}q^{\binom{m}{2}}[m]_q! = (q^m-1)(q^{m-1}-1)\cdots(q^2-1)\]
also equals zero, which happens precisely when $q$ is a $d$th root of unity for some~$1 < d \leq m$.

Finally, before stating the classification, we observe that thanks to Proposition~\ref{prop:qbinom} a map $\varphi\colon\mathcal{R}_q^{+}\to\mathbf{k}$ is determined by where it sends $q$ and~$\qbinom{x}{m}$ for all~$m \in \mathbb{N}$. And thanks to Proposition~\ref{prop:localization}, a map $\varphi\colon\mathcal{R}_q\to\mathbf{k}$ is determined by this same information as well.

\begin{theorem} \label{thm:mapstofields}
Let $\mathbf{k}$ be a field. Then the ring homomorphisms~$\varphi\colon\mathcal{R}_q^{+}\to\mathbf{k}$ are exactly the following:
\begin{enumerate}
\item $\boldsymbol{q=0}$: For each choice of $k \in \mathbb{N} \cup \{\infty\}$, we have a map $\varphi$ defined by
\begin{align*}
\varphi(q) &:= 0; \\
\varphi\left(\qbinom{x}{m}\right) &:= \begin{cases}1 &\textrm{if $m \leq k$}, \\ 0 &\textrm{otherwise} \end{cases} \qquad \textrm{for each $m \in \mathbb{N}$}.
\end{align*}
\item {\bf $\boldsymbol{q}$ a root of unity}: For each choice of positive integer~$d$, $\omega \in \mathbf{k}$ a primtive $d$th root of unity, $n_0 \in [0,d-1]$, and $t$ either any element of the field $t \in \mathbf{k}$ if~$\mathrm{char}(\mathbf{k}) = 0$ or a $p$-adic integer $t\in\mathbb{Z}_p$ if $\mathrm{char}(\mathbf{k}) = p > 0$, we have a map $\varphi$ defined by
\begin{align*}
\varphi(q) &:= \omega; \\
\varphi\left(\qbinom{x}{m}\right) &:= \binom{\frac{t-n_0}{d}}{m'}\qbinom{n_0}{m_0}_{q:=\omega} \, \parbox{2.5in}{\begin{center}for each $m \in \mathbb{N}$, where $m = dm' + m_0$ with $m_0 \in [0,d-1]$.\end{center}}
\end{align*}
\item {\bf $\boldsymbol{q}$ not zero, not a root of unity}: For each choice of $\kappa \in \mathbf{k}$ not equal to zero and not a root of unity, and $t \in \mathbf{k}$, we have a map $\varphi$ defined by
\begin{align*}
\varphi(q) &:= \kappa; \\
\varphi\left(\qbinom{x}{m}\right) &:= \qbinom{t}{m}_{q:=\kappa} \qquad \textrm{for each $m \in \mathbb{N}$}.
\end{align*}
\end{enumerate}
The maps $\varphi\colon \mathcal{R}_q \to \mathbf{k}$ are the same as the above, except that Case 1 (where $q$ is sent to zero) does not occur.
\end{theorem}

\begin{proof}
{\bf Case 1 $\boldsymbol{(q=0)}$}: Not surprisingly, this is the most degenerate case. First let us show that the $\varphi$ described indeed define ring homomorphisms. So let $\varphi$ be as in the statement. If~$k \neq \infty$ then $\varphi = \text{std}_{k,0}$, so it is indeed a homomorphism.  If~$k = \infty$ then for any $P(x)\in \mathcal{R}_q^+$ we have $\varphi(P(x)) = \text{std}_{N,0}(P(x))$ for all sufficiently large~$N$ (how large $N$ needs to be depends on the degree of $P(x)$), which in particular implies $\varphi$ is a homomorphism.

To see these are all ring homomorphisms sending $q$ to zero note that if we specialize the formula in Theorem~\ref{thm:qbinommult} to~$q := 0$ we get that for $j\leq k$ integers:
\[\qbinom{x}{j}\qbinom{x}{k} = \qbinom{x}{k}.\]
Setting $j=k$ we see that $\qbinom{x}{k}$ must get sent to $0$ or $1$ for all $k$.  Then this formula tells us that if we send $\qbinom{x}{j}$ to $0$ then we must send $\qbinom{x}{k}$ to $0$ for all $k >j$.  From there it is easy to see that any such homomorphism must agree with one on our list, as they are completely determined by how many of the $q$-binomial coefficient polynomials get sent to $1$.

{\bf Case 2 ($\boldsymbol{q}$ a root of unity)}: This is really the interesting case. Let us break it into two subcases, based on whether or not $q$ is sent to~$1$.

{\bf Case 2(a) ($\boldsymbol{q}=1$)}: Note that a homomorphism from $\mathcal{R}_q^{+}$ to $\mathbf{k}$ where~$q$ is sent to~$1$ is the same thing as a homomorphism from $\mathcal{R}$ to~$\mathbf{k}$. These are essentially characterized by Brizolis's result, and this formulation of the classification appears in in~\cite{harman2015stability}. For completeness we repeat the argument presented there. 

If $\mathbf{k}$ is in characteristic zero then as we mentioned earlier, the binomial coefficient polynomials naturally sit inside $\mathbf{k}[x]$, so for any point $t \in \mathbf{k}$ we can just evaluate each polynomial at $t$ to get a homomorphism into $\mathbf{k}$. Moreover any ring homomorphism from $\mathcal{R}$ to a field of characteristic zero  is completely defined by the value that $x$ gets sent to, so we get that $\mathrm{Hom}(\mathcal{R},\mathbf{k}) \simeq \mathbf{k}$ and the standard evaluation maps correspond to the copy of $\mathbb{N}$ in $\mathbf{k}$.
 
In characteristic $p$ things are somewhat different as the binomial coefficient polynomials do not naturally sit inside $\mathbf{k}[x]$. Let us look again at Lucas' theorem:
\[\binom{n}{m} \equiv \binom{n_0}{m_0} \binom{n_1}{m_1} \binom{n_2}{m_2} \dots \binom{n_k}{m_k}  \mod p.\]
For fixed $m$ this formula only depends on the first $k$ base $p$ digits of $n$. Hence we can naturally evaluate modulo $p$ the binomial coefficient polynomials at any $p$-adic integer $t$ (since they still have a base $p$ expansion) and obtain evaluation maps $\mathrm{ev}_t: \mathcal{R} \rightarrow \mathbb{F}_p$ for each $t \in \mathbb{Z}_p$. These $\mathrm{ev}_t$ are exactly the~$\varphi$ described in the statement. The following lemma completes the characterization in this case.

\begin{lemma}\label{specr}
Any homomorphism $\varphi: \mathcal{R} \rightarrow \mathbf{k}$ of rings from $\mathcal{R}$ into a field $\mathbf{k}$ of characteristic $p$ factors as $\mathrm{ev}_t$ for some $t \in \mathbb{Z}_p$, followed by the inclusion of $\mathbb{F}_p$ into $\mathbf{k}$.

\end{lemma}

\begin{proof}[Proof of lemma] First consider maps $\varphi$ to $\mathbb{F}_p$. Let $m = m_0 + m_1p + \dots + m_kp^k$ be the base $p$ expansion of a positive integer $m$.  Lucas' theorem tells us that the polynomial
\[ F(x) =  \frac{\binom{x}{m} - \binom{x}{m_0} \binom{\binom{x}{p}}{m_1} \binom{\binom{x}{p^2}}{m_2} \dots \binom{\binom{x}{p^k}}{m_k}}{p}\]
is integer-valued, and hence $pF(x)$ gets sent to zero under $\varphi$.  This implies that the images of $\binom{x}{m}$ for all $m$ are determined by the images of $x, \binom{x}{p}, \binom{x}{p^2}, \dots$  We may then interpret these values as the base $p$ digits of some $t \in \mathbb{Z}_p$ and conclude that $\varphi = \mathrm{ev}_t$ since they agree on a basis for $\mathcal{R}$.

To see any map into an arbitrary field of characteristic $p$ must factor through a map to~$\mathbb{F}_p$ note that for any $P(x) \in \mathcal{R}$ the image under $\varphi$ of $P(x)^p - P(x)$ in~$\mathbf{k}$ is~$p$ times the image of $\frac{P(x)^p-P(x)}{p} \in \mathcal{R}$. Therefore the image of $P(x)$ is fixed by the Frobenius map and hence is in $\mathbb{F}_p$. 
\end{proof}

{\bf Case 2(b) ($\boldsymbol{q}$ a primitive $\boldsymbol{d}$th root of unity, $\boldsymbol{d>1}$)}:  First note that if~$\mathrm{char}(\mathbf{k}) = p >0$ then necessarily $p \nmid d$ and thus $\frac{t-n_0}{d} \in \mathbb{Z}_p$ since $d$ is a unit in~$\mathbb{Z}_p$. So the term $\binom{\frac{t-n_0}{d}}{m'}$ is well-defined according to our earlier definition.

Now note that if $t = n$ is an integer congruent to $n_0$ modulo $d$ then the $q$-Lucas theorem tells us that the map $\varphi$ described is the standard map~$\mathrm{std}_{n,\omega}$ which we know to be a ring homomorphism. It then follows that the (a priori just linear) map $\varphi$ is a ring homomorphism for all appropriate values of $t$, as the definition of the map varies algebraically in $t$ and the set of nonnegative integers congruent to $n_0$ modulo $d$ is dense with respect to the Zariski topology on $\mathbf{k}$ in characteristic zero, as well as with respect to the $p$-adic topology on $\mathbb{Z}_p$.

We need to show that these are all such homomorphisms. Note that any homomorphism from $\mathcal{R}^{+}_q$ to $\mathbf{k}$ sending $q$ to a primitive $d$th root of unity must send~$x \in \mathcal{R}^{+}_q$ to one of $d$ possible values: $[0]_q, [1]_q, \dots, \text{\ or \ } [d-1]_q$. This is since the following identity holds in $\mathcal{R}^{+}_q$:
\[x(x-[1]_q)(x-[2]_q)\dots(x-[d-1]_q) = [d]_q! q^{\binom{d}{2}} \qbinom{x}{d}.\]
We see that indeed if we send $q$ to a $d$th root of unity the right hand side vanishes and therefore we are forced to send $x$ to one of these values in order for the left hand side to vanish.  This value where $x$ gets sent corresponds to the discrete parameter $n_0$ in the statement of the theorem.

Next, we note that if we precompose our map into a field with the quantum Frobenius map $\Psi_d: \mathcal{R} \to \mathcal{R}^{+}_q / {\Phi_d(q)}$ then we obtain a homomorphism from~$\mathcal{R}$ into~$\mathbf{k}$ and may apply the $q=1$ classification. This map from $\mathcal{R}$ into $\mathbf{k}$ corresponds to the choice of $t$ (where the labeling variable $t$ is shifted by the invertible map $t\mapsto \frac{t-n_0}{d}$).

So in order to complete the classification in this case we just need to show that any such homomorphism into a field is completely determined by the data of where $x$ gets sent, and its restriction to the image of $\mathcal{R}$ under $\Psi_d$. So it is enough to show that we can express $\qbinom{x}{k}$ as a polynomial of $x$ and elements of the image of $\Psi_d$, with coefficients in $\mathbb{Z}[q]$ localized at $\Phi_d(q)=0$. Write $k = dk'+k_0$ with $k_0 \in [0,d-1]$.

If $k'=0$ then the usual formula for $\qbinom{x}{k}$ has denominator relatively prime to $\Phi_d(q)$, and hence it can be expressed just in terms of $x$ after localizing. Otherwise, consider the product:
\[\qbinom{x}{k_0}\qbinom{x}{dk'} = \sum_{i = dk'}^{k}\qbinom{x}{i} \qbinom{i}{dk'}_q \qbinom{dk'}{k-i}_q q^{(i-k_0)(i-dk')}.\]
The $q$-Lucas' theorem tells us that the $\qbinom{dk'}{k-i}_q$ term is zero unless $i = k$. Simplifying this remaining term using the $q$-Lucas' theorem we obtain:
\[\qbinom{x}{k_0}\qbinom{x}{dk'}= \qbinom{x}{k}.\]
So we see that indeed we can express $\qbinom{x}{k}$ in terms of $x$ and elements of the image of $\Psi_d$, finishing this case of the classification.

{\bf Case 3 ($\boldsymbol{q}$ not zero, not a root of unity)} Since the denominators of the $q$-binomial coefficient polynomials are products of $q$ and cyclotomic polynomials in $q$ it is clear that $R_q^+$ localized away from $q$ being zero or a root of unity is just a polynomial ring in $x$ over $\mathbb{Z}[q]$ localized away from $q$ being zero or a root of unity. Hence such maps to fields are just given by a (nonzero, non-root of unity) choice $\kappa$ of where to send $q$, and a choice $t$ of where to send $x$.
\end{proof}

\end{section}

\begin{section}{Open questions and future directions} \label{sec:future}

In this section we discuss some open questions and possible future directions in the investigation of the ring~$\mathcal{R}_q$.

\subsection{A dilation operator}
For any integer $m \geq 1$, consider the \emph{dilation operator} $D_m \colon \mathbb{Q}(q)[x]\to \mathbb{Q}(q)[x]$ given by
\[D_m(x) := \frac{ ((q-1)x+1)^m-1}{q-1} = \sum_{i=1}^{m}\binom{m}{i}(q-1)^mx^m\]
and extended $\mathbb{Q}(q)$-linearly. This operator may be easier to understand in the ``$z$-variable'' formulation discussed in Section~\ref{sec:relrings}: with respect to this generator it is defined by $D_m(z) := z^m$. This operator is defined so that for $P(x) \in \mathbb{Q}(q)[x]$ and $n \in \mathbb{Z}$ we have $D_mP([n]_q) = P([mn]_q)$. Thus $D_m$ restricts to a $\mathbb{Z}[q,q^{-1}]$-linear homomorphism $D_m\colon \mathcal{R}_q\to \mathcal{R}_q$. We can of course ask how~$D_m$ acts on the basis of $q$-binomial coefficient polynomials. That is, let us try to understand the coefficients $\delta_{m,i,k}(q) \in \mathbb{Z}[q,q^{-1}]$ when we write
\[D_m\qbinom{x}{i} = \sum_{k=0}^{\infty} \delta_{m,i,k}(q)\qbinom{x}{k}.\]
These coefficients $\delta_{m,i,k}(q)$ are the first time that ``positivity'' seems to fail for~$\mathcal{R}_q$. Note that~$D_m$ is a map of degree~$m$ in~$x$. Thus~$\delta_{m,i,mi}(q) \neq 0$. However, when we specialize~$q:=1$, this map $D_m$ becomes a map of degree one: namely,~$D_m(x) := mx$. And so we have~$(q-1) \mid \delta_{m,i,k}(q)$ for all $k > i$. More generally by the same reasoning we have~$(q-1)^{\lfloor (k-1)/i \rfloor} \mid \delta_{m,i,k}(q)$. At any rate we certainly do not have that~$\delta_{m,i,k}(q) \in \mathbb{N}[q]$. But even accounting for this predictable power of $(q-1)$, positivity for these $\delta_{m,i,k}(q)$ can apparently fail for other reasons. For example, computation with Sage mathematical software~\cite{sage} tells us that
\begin{gather*}
D_2\qbinom{x}{3} = (q + 1)(q^2 + 1)\qbinom{x}{2} + q(q + 1)(q^2 + 1)(q^5 + q^3 + q^2 - 1)\qbinom{x}{3} \\
+q^5(q - 1)(q + 1)(q^2 + 1)(q^2 + q + 1)(q^4 + q^2 + q + 1)\qbinom{x}{4} \\
+q^7(q - 1)^2(q + 1)^2(q^2 - q + 1)(q^2 + 1)(q^2 + q + 1)(q^4 + q^3 + q^2 + q + 1)\qbinom{x}{5}\\
+q^{12}(q - 1)^3(q + 1)^2(q^2 - q + 1) (q^2 + 1) (q^2 + q + 1) (q^4 + q^3 + q^2 + q + 1)\qbinom{x}{6}.
\end{gather*}
The fact that $\delta_{2,3,3}(q) \notin \mathbb{N}[q]$ is especially troubling.

It is worth contrasting the above discussion with the fact that we do have positivity for the coefficients $\delta_{m,i,k}(1)$ when we specialize $q:=1$. To this end, we observe the (undoubtedly folklore result) that
\begin{equation} \label{eq:binomsplit}
\binom{xy}{i} = \sum_{j,k=0}^{\infty} \partial_{j,k}^{i} \binom{x}{j}\binom{y}{k}
\end{equation}
where $\partial_{j,k}^{i}$ is the number of $j \times k$ $(0,1)$-matrices with exactly $i$ ones and no row or column of all zeros. Thanks to Lemma~\ref{lem:interpolation}, equation~(\ref{eq:binomsplit}) can be proved by taking $x:=n,y:=m$ with $n,m \in \mathbb{N}$ and interpreting both sides as the number of subsets of size $i$ of the set~$\{(a,b)\colon a \in \{1,\ldots,n\}, b \in\{1,\ldots,m\}\}$: it is obvious why the left-hand side counts these subsets; the right-hand side counts these subsets by grouping them according to their projections to the first and second components. We can specialize $y:=m$ in~(\ref{eq:binomsplit}) and conclude that
\[\delta_{m,i,k}(1) = \sum_{j=0}^{\infty} \partial_{j,k}^{i} \binom{m}{j}.\]
So in particular $\delta_{m,i,k}(1) \in \mathbb{N}$. Note that unlike other coefficients studied in this paper, there does not appear to be any simple product formula for the~$\partial_{j,k}^{i}$. Therefore, even in this $q=1$ case where we have positivity for these dilation coefficients, computing~$\delta_{m,i,k}(1)$ seems hard.

\subsection{Intersection of \texorpdfstring{$\mathcal{R}_q^{+}$}{Rq+} and \texorpdfstring{$\mathcal{R}_q^{-}$}{Rq-}}

What can we say about the ring $\mathcal{R}_{q}^{+} \cap \mathcal{R}_{q}^{-}$? Note that $\mathcal{R}_{q}^{+} \cap \mathcal{R}_{q}^{-}$ is naturally only a $\mathbb{Z}$-algebra, that is, a ring. Of course $\mathbb{Z} \subseteq \mathcal{R}_{q}^{+} \cap \mathcal{R}_{q}^{-}$. But we can say more: by Proposition~\ref{prop:invconsts} we at least have $\qbinom{x}{k}, \overline{\qbinom{x}{k}} \in \mathcal{R}_{q}^{+} \cap \mathcal{R}_{q}^{-}$ for all~$k \in \mathbb{N}$. It seems unlikely that $\mathcal{R}_{q}^{+} \cap \mathcal{R}_{q}^{-} = \mathrm{Span}_{\mathbb{Z}}\{\qbinom{x}{k},\overline{\qbinom{x}{k}}\colon k \in \mathbb{N}\}$ as a~$\mathbb{Z}$-module . Indeed, it is not even clear that the set~$\{\qbinom{x}{k},\overline{\qbinom{x}{k}}\colon k \in \mathbb{N}\}$ generates~$\mathcal{R}_{q}^{+} \cap \mathcal{R}_{q}^{-}$. Can it be shown that $\mathcal{R}_{q}^{+} \cap \mathcal{R}_{q}^{-}$ is not a free $\mathbb{Z}$-module? In general this ring~$\mathcal{R}_{q}^{+} \cap \mathcal{R}_{q}^{-}$ seems quite mysterious to us. More generally, for any~$m,m' \in \mathbb{Z}$ we can ask what the intersection~$\mathcal{R}_{q}^{+,m} \cap \mathcal{R}_{q}^{-,m'}$ looks like. Recall that the ring~$\mathcal{R}_q^{+,m}$ is defined in Section~\ref{sec:shift}.

\subsection{A Hopf algebra?}

The polynomial ring $\mathbb{Q}[x]$ can be given the structure of a Hopf algebra over~$\mathbb{Q}$ by defining the comultiplication as $\Delta(x^n) = \sum_{k=0}^{n}\binom{n}{k}x^{k}\otimes x^{n-k}$, the counit as~$\varepsilon(x^n) = \begin{cases} 1 &\textrm{if $n=0$},\\ 0 &\textrm{otherwise}.\end{cases}$, and the antipode as~$x^n \mapsto (-1)^n x^n$. With this coalgebra structure on $\mathbb{Q}[x]$, we have for all $k \in \mathbb{N}$ that
\begin{align*}
\Delta\left( \binom{x}{k} \right) &= \sum_{i=0}^{k}\binom{x}{i}\otimes\binom{x}{k-i}; \\
\varepsilon\left(\binom{x}{k} \right) &= \begin{cases} 1 &\textrm{if $k=0$},\\ 0 &\textrm{otherwise}.\end{cases}
\end{align*}
The above formulae define a coalgebra structure on~$\mathcal{R}$ which is called the ``divided power coalgebra'' (see~\cite[Example 1.1.4(2)]{dascalescu2001hopf}). In particular the Hopf algebra structure on~$\mathbb{Q}[x]$ restricts to a Hopf algebra structure on~$\mathcal{R}$ (which is, however, now a Hopf algebra over $\mathbb{Z}$, not $\mathbb{Q}$). Note that the antipode for~$\mathcal{R}$ viewed as a Hopf algebra in this way is the ``bar involution'' (specialized to~$q:=1$, of course). It would be interesting to define a Hopf algebra structure on~$\mathcal{R}_q$ for which the bar involution is the antipode. One immediate issue is that~$\mathcal{R}_q$ is naturally a~$\mathbb{Z}[q,q^{-1}]$-algebra, but the bar involution is not a  $\mathbb{Z}[q,q^{-1}]$-linear map: rather, it ``twists'' the coefficient ring. Perhaps there is some way to relax the conditions of a Hopf algebra to only require the antipode be a semi-linear map.

\begin{subsection}{Maximal ideals of \texorpdfstring{$\mathcal{R}_q$}{Rq}}
As mentioned at the beginning of Section~\ref{sec:mapstofields}, Brizolis~\cite{brizolis1976ideals} offered a very nice classification of the maximal ideals of~$\mathcal{R}$. Considering the classification of maps from $\mathcal{R}_q$ into a field we provide (Theorem~\ref{thm:mapstofields}), one might be optimistic that we could find a similar classification of maximal ideals of $\mathcal{R}_q$. As we explained earlier, this would amount to determining when a map from $\mathcal{R}_q$ to a field is surjective. Note that a consequence of Brizolis's classification is that if $\mathbf{k}$ is a field and $\varphi\colon \mathcal{R} \to \mathbf{k}$ is surjective, then $\mathbf{k} = \mathbb{F}_p$ for some prime~$p$. In particular there is no surjective map from $\mathcal{R}$ to a field of charactersitic zero. In contrast,~$\mathcal{R}_q$ actually does have surjective maps to fields of characteristic zero. For example, consider the map~$\varphi\colon \mathcal{R}_q\to\mathbb{Q}$ defined by $\varphi(q) := \frac{1}{2}$ and $\varphi(x) := 2$. It turns out that~$\varphi$ is surjective. Indeed, observe that for any $k \in \mathbb{N}$,
\begin{align*}
\varphi\left(q^{\binom{k}{2}+1}\qbinom{x}{k}\right) &= \left(\frac{1}{2}\right)^{\binom{k}{2}+1} \, \frac{2(2-1)(2-\frac{3}{2})\cdots(2-\frac{2^k-1}{2^{k-1}})}{\left(\frac{1}{2}\right)^{\binom{k}{2}}\cdot 1 \cdot \frac{3}{2}\cdot\frac{7}{4}\cdots\frac{2^k-1}{2^{k-1}}} \\
&= \prod_{i=1}^{k} \frac{1}{2^i-1}.
\end{align*}
Then note that for any prime $p$, there is some $k$ such $p\mid 2^k-1$, just because~$2$ has to have some mutliplicative order in $\mathbb{F}_p$. Thus we see that $\frac{1}{p}$ belongs to the image~$\varphi(\mathcal{R}_q)$ for every prime $p$. But if $\frac{1}{p} \in \varphi(\mathcal{R}_q)$ for all primes $p$, then clearly~$\varphi(\mathcal{R}_q)=\mathbb{Q}$ as claimed. Considering the fact that some maps~$\varphi\colon \mathcal{R} \to \mathbb{Q}$ are surjective, while others are certainly not (such as any with~$\varphi(q)=1$), it seems that the general problem of determining when a map from~$\mathcal{R}_q$ to a field is surjective could involve some delicate number theory. Thus while it would certainly be interesting to classify all maximal ideals of~$\mathcal{R}_q$, we doubt that there is as nice a classification as Brizolis's classification of maximal ideals of~$\mathcal{R}$.
\end{subsection}

\end{section}

\printbibliography

\end{document}